\documentclass[10pt]{amsart}
\usepackage{amsmath}
\usepackage{amssymb}
\usepackage{amsfonts}
\usepackage{mathrsfs}
\usepackage{graphicx}
\usepackage{epsfig}
\usepackage[T1]{fontenc}
\numberwithin{equation}{section}       

\theoremstyle{plain}
\newtheorem{theorem}{Theorem}[section]

\newtheorem{prop}{Proposition}[section]

\newtheorem{coro}[prop]{Corollary}
\newtheorem{lemma}[prop]{Lemma}

\newtheorem*{mainthm}{Main Theorem}
\newtheorem*{Ledthm}{Ledrappier's Theorem}

\theoremstyle{definition}

\theoremstyle{remark}

\newtheoremstyle{citing}
  {3pt}
  {3pt}
  {\itshape}
  {}
  {\bfseries}
  {.}
  {.5em}
  {\thmnote{#3}}

\theoremstyle{citing}

\DeclareMathAlphabet{\mathpzc}{OT1}{pzc}{m}{it} 

\newcommand{\R}{\mathbb{R}}

\newcommand{\Z}{\mathbb{Z}}

\newcommand{\teta}{\widetilde{\teta}}

\newcommand{\eps}{\varepsilon}

\newcommand{\dist}{d}

\renewcommand{\=}{ : = }

\begin{document}
\title[]{Hausdorff dimension of the graphs of the classical Weierstrass functions}
\author{Weixiao Shen}
\address{Shanghai Center for Mathematical Sciences, Fudan University, 220 Handan Road, Shanghai, China 200433}
\email{wxshen@fudan.edu.cn}
\date{\today}

\begin{abstract}
We show that the graph of the classical Weierstrass function $\sum_{n=0}^\infty \lambda^n \cos (2\pi b^n x)$ has Hausdorff dimension $2+\log\lambda/\log b$, for every integer $b\ge 2$ and every $\lambda\in (1/b,1)$. Replacing $\cos(2\pi x)$ by a general non-constant $C^2$ periodic function, we obtain the same result under a further assumption that $\lambda b$ is close to $1$.
\end{abstract}
\maketitle
\section{Introduction}
In this paper, we study the Hausdorff dimension of the graph of the following Weierstrass function
$$W_{\lambda,b}(x)=\sum_{n=0}^\infty \lambda^n \cos ( 2\pi b^n x), x\in\R$$
where $0<\lambda<1<b$ and $b\lambda>1$. These functions, studied by Weierstrass and Hardy ~\cite{Hardy}, are probably the most well-known examples of continuous but nowhere differentiable functions. Study of the graph of these and related functions from a geometric point of view as fractal sets have attracted much attention since Besicovitch and Ursell~\cite{BU}. A long standing conjecture asserts that the Hausdorff dimension of the graph of $W_{\lambda, b}$ is equal to $$D=2+\frac{\log\lambda}{\log b},$$ see for example~\cite{Mandel}. Although the box dimension and packing dimension have been shown to be equal to $D$ for a large class of functions including all the functions $W_{\lambda,b}$ (see~\cite{HL, KMY, R}), the conjecture about Hausdorff dimension remains open even in the case when $b$ is an integer.

\begin{mainthm}
For any integer $b\ge 2$ and any $\lambda\in (b^{-1},1)$, the Hausdorff dimension of the graph of the Weierstrass function $W_{\lambda,b}$ is equal to $D.$
\end{mainthm}

More generally, we consider the following function:
$$f_{\lambda, b}^\phi(x)=\sum_{n=0}^\infty \lambda^n \phi (b^n x),$$
where $\phi$ is a $\Z$ -periodic function and $\lambda, b$ are as above. So $W_{\lambda,b}$ corresponds to the case $\phi(x)=\cos (2\pi x)$. Our method also shows the following:

\begin{theorem} \label{thm:gammalarge}
For any $\Z$-periodic, non-constant $C^2$ function $\phi:\R\to \R$ and any integer $b\ge 2$ there exists $K_0=K_0(\phi, b)>1$ such that if $1<\lambda b< K_0$,  then the graph of $f_{\lambda, b}^\phi$ has Hausdorff dimension $D$.
\end{theorem}

Recently, Bara\'nsky, B\'ar\'any and Romanowska~\cite{BBR}, based on results of Ledrappier~\cite{L} and Tsujii~\cite{Tsu}, proved that for each integer $b\ge 2$, there is a number $\lambda_b\in (0,1)$ such that the Hausdorff dimension of the graph of $W_{\lambda, b}$ is equal to $D$ provided that $\lambda_b<\lambda<1$. Furthermore, given an integer $b\ge 2$, they proved that the graph of $f_{\lambda, b}^\phi$ has Hausdorff dimension $D$ for generic $(\lambda, \phi)$. We refer to~\cite{BBR} for other progress on this and related problems.
In order to prove our theorems, we have to introduce and verify a modified version of a transversality condition in ~\cite{Tsu} for all the cases.
The proof of Theorem~\ref{thm:gammalarge} also uses some results of~\cite{BKRU}.

The assumption that $b$ is an integer enables us to approach the problem using dynamical systems theory. Indeed, in this case, the graph of $f_{\lambda, b}^\phi$ can be interpreted as an invariant repeller for the expanding dynamical system $\Phi: \R/\Z\times \R\to \R/\Z\times \R$,
$$\Phi(x,y)=\left( bx\mod 1, \frac{y-\widehat{\phi} (x)}{\lambda}\right),$$
where $\widehat{\phi}:\R/\Z\to \R$ denote the map naturally induced by $\phi$. By method of ergodic theory of smooth dynamical systems,  Ledrappier~\cite{L}
reduced the problem on Hausdorff dimension of the graph of $f_{\lambda, b}^\phi$ to the study of local dimension of the measures $m_x$ defined below.

Let $\mathcal{A}=\{0,1,\ldots, b-1\}$, and consider the Bernoulli measure $\mathbb{P}$ on $\mathcal{A}^{\Z^+}$ with uniform probabilities $\{1/b, 1/b, \ldots, 1/b\}^{\Z^+}$. For $x\in\R$ and $\textbf{u}=\{u_n\}_{n=1}^\infty\in\mathcal{A}^{\Z^+}$, define \begin{equation}\label{eqn:Sxi}
S(x,\textbf{u})=\sum_{n=1}^\infty \gamma^{n-1} \psi\left(\frac{x}{b^n}+\frac{u_1}{b^n}+\cdots+\frac{u_n}{b}\right),
\end{equation}
where
\begin{equation}\label{eqn:gamma}
\gamma=\frac{1}{\lambda b}\text{ and }
\psi (x)=\phi'(x).
\end{equation}
These functions are, up to some multiplicative constant, the slope of the strong unstable manifolds of the expanding endomorphism $\Phi$.
For each $x\in \R$, let $m_x$ denote the Borel probability measure in $\R$ obtained as pushforward of the measure $\mathbb{P}$ by the function $\textbf{u}\mapsto S(x,\textbf{u})$.

We say that a Borel measure $\mu$ in a metric space $X$ has local dimension $d$ at a point $x\in X$, if $$\lim_{r\to 0} \frac{\log \mu(B_r(x))}{\log r} =d,$$
where $B_r(x)$ denotes the ball of radius $r$ centered at $x$. If the local dimension of $\mu$ exists and is equal to $d$ at $\mu$-a.e. $x$, then we say that $\mu$ has local dimension $d$ and write $\dim(\mu)=d$. It is well-known that if $\mu$ has local dimension $d$, then any Borel set of positive measure has Haudorff dimension at least $d$.

\begin{Ledthm}
Let $\phi:\R \to \R$ be a continuous, piecewise $C^{1+\alpha}$ and $\Z$-periodic function.
Assume that $\dim(m_x)=1$ holds for Lebesgue a.e. $x\in (0,1)$.
Then the Hausdorff dimension of the graph of $f_{\lambda, b}^\phi$ is equal to $D$.
\end{Ledthm}

To prove this theorem, Ledrappier studied the local dimension of the measure
$\mu=\mu_{\lambda,b}^\phi$ obtained as the lift of the Lebesgue measure on $[0,1]$ to the graph of $f_{\lambda, b}^{\phi}$.
Combining results of Ledreppier and Young~\cite{LY} with a variation of Marstrand's projection theorem, Ledrappier proved that $\dim(\mu)=D$, provided that $\dim(m_x)=1$ holds for Lebesgue almost every $x$. This proves that the Hausdorff dimension of the graph of $f_{\lambda,b}^\phi$ is at least $D$. As it is easy to see that the box dimension is at most $D$, the theorem follows. For the convenience of the readers not familiar with~\cite{LY}, we include a self-contained elementary proof of Ledrappier's Theorem in the appendix (assuming $\phi'$ has no discontinuity for simplicity). The proof is of course motivated by the original proof in~\cite{L}, but we also borrowed ideas in~\cite{K} where Keller gives an alternative proof of a weak version of Ledrappier's theorem.  Keller's version is indeed enough for our purpose, although he used notation quite different from us.

Clearly, if $m_x$ is absolutely continuous with respect to the Lebesgue measure on $\R$, then $\dim (m_x)=1$.
The case when $\phi(x)=\dist (x, \Z)$ and $b=2$ is a famous problem in harmonic analysis and was studied first in \cite{E}. In this case,
the absolute continuity of $m_x$ was established in ~\cite{S} for almost every $\gamma\in (1/2,1)$. See also~\cite{PS}.
In general, $m_x$'s are the conditional measures along vertical fibers of the unique SRB measure $\vartheta=\vartheta_{b,\gamma}^{\psi}$ of the skew product map $T:\R/\Z\times \R \to \R/\Z\times \R$,
\begin{equation}\label{eqn:mapT}
T(x,y)=\left(bx \mod 1, \gamma y+\widehat{\psi}(x)\right),
\end{equation}
where $\psi(x)$ and $\gamma$ are as above.
The map $T$ is an Anosov endomorphism and uniformly contracting along vertical fibers. The graph of the functions $x\mapsto S(x,\textbf{u})$ are the unstable manifolds. In~\cite{Tsu}, Tsujii posed some condition on the transversality of these unstable manifolds and showed that this condition implies absolute continuity of $m_x$ for almost every $x$ (and the absolutely continuity of the SRB measure $\vartheta$). Furthermore, for given $b$, he verified his condition for generic $(\gamma,\psi)$.

However, for given $\psi$ it is not easy to verify Tsujii's condition, if possible at all. In fact, it was a major step in the recent work~\cite{BBR} to verify that Tsujii's condition holds for $\psi(x)=-2\pi \sin (2\pi x)$ when $\lambda\in (\lambda_b,1)$. We shall show in Section 3 that Tsujii's condition is indeed satisfied when $b\ge 6$ for this particular $\psi$ and all $\lambda\in (1/b,1)$ (or equivalently, all $\gamma \in (1/b,1)$). To deal with the case $2\le b\le 5$, we shall pose a modified version of Tsujii's condition. We shall show that the new (weaker) condition is still enough to guarantee absolute continuity of $m_x$ for Lebesgue a.e. $x$. Then we verify this new condition and conclude the proof of the Main Theorem by Ledrappier's Theorem.

\begin{theorem}\label{thm:cossole}
Let $b\ge 2$ be an integer, let $\gamma\in (1/b,1)$ and let $\psi=-2\pi \sin (2\pi x)$.
Then the SRB measure $\vartheta$ for the map $T$ is absolutely continuous with respect to the Lebesgue measure on $\R/\Z\times \R$ and with a square integrable density. In particular, for Lebesgue a.e. $x\in\R$, the measure $m_x$ defined above is absolutely continuous with respect to Lebesgue measure and with a square integrable density.
\end{theorem}

In the next section, we modify Tsujii's transversality condition. In particular, we shall define a new number $\sigma (q)$ to replace the number $e(q)$ in Tsujii's work. We shall prove Theorem~\ref{thm:gammalarge} and state the plan of  the proof of Theorem~\ref{thm:cossole} in that section. Sections 2-5 are devoted to the proof of Theorem~\ref{thm:cossole}. In the appendix, Section 6, we provide a proof of Ledrappier's theorem.

\medskip
{\bf Acknowledgment.} I would like to thank D. Feng, W. Huang and J. Wu for drawing my attention to the recent work \cite{BBR}. I would also like to thank H. Ruan and Y. Wang for reading carefully a first version of the manuscript and pointing out a number of errors.

\section{Tsujii's transversality condition on fat solenoidal attractors}
In this section, we study a map $T$ of the form (\ref{eqn:mapT}), where $b\ge 2$ is an integer, $b^{-1}<\gamma<1$ and $\psi$ is a $\Z$-periodic $C^1$ function.
These maps were studied in~\cite{Tsu} from measure-theoretical point of view, and in~\cite{BKRU} from topological point of view. In~\cite{Tsu}, Section 2, it was shown that $T$ has a unique SRB measure $\vartheta$, for which Lebesgue almost every point $(x,y)$ in $\R/\Z\times \R$ is a generic point, i.e.
$$\frac{1}{n}\sum_{i=0}^{n-1}\delta_{T^i(x,y)}\to \vartheta \text{ as } n\to\infty,$$
in the weak star topology, where $\delta_\cdot$ denote the Dirac measure. The measure $\vartheta$ has an explicit expression through the measures $m_x$ defined in the introduction: identifying $\R/\Z$ with $[0,1)$ in the natural way, for each Borel set $B\subset \R/\Z\times \R$,
$$\vartheta(B)=\int_0^1 m_x(B_x) dx,$$
where $B_x=\{y\in\R: (x,y)\in B\}.$
We are interested in the absolute continuity of the SRB measure $\vartheta$, or equivalently, the absolute continuity of $m_x$ for Lebesgue almost every $x$.
In ~\cite{Tsu}, Tsujii posed some condition on the transversality of the graphs of the functions $S(x,\textbf{u})$ (which are understood as unstable manifolds of $T$) which guarantees the absolute continuity of $\vartheta$.

In this section, we introduce a modified version of Tsujii's condition and show that the weaker condition already implies absolute continuity of $\vartheta$.  We shall prove Theorem~\ref{thm:gammalarge} by verifying the modified condition.

\medskip
{\bf Notation.} For each $x\in\R$ and $(u_1u_2\cdots u_q)\in \mathcal{A}^q$, let $$x(\textbf{u})=\frac{x+ u_1 +u_2 b+\cdots + u_q b^{q-1}}{b^q}.$$
We use $S'(x,\textbf{u})$ to denote the derivative of $S(x,\textbf{u})$ regarded as a function of $x$.

\subsection{Transversality}
We say that two words $\textbf{i},\textbf{j}\in \mathcal{A}^{Z^+}$ are {\em $(\eps,\delta)$-transversal} at a point $x_0\in\R$  if one of the following holds:
$$|S(x_0, \textbf{i})-S(x_0,\textbf{j})|> \eps \text{ or }
\left|S'(x_0, \textbf{i})-S'(x_0,\textbf{j})\right|> \delta.$$
Otherwise, we say that $\textbf{i}$ and $\textbf{j}$ are {\em $(\eps,\delta)$-tangent at $x_0$.}
Let $E(q, x_0; \eps,\delta)$ denote the set of pairs $(\textbf{k},\textbf{l})\in\mathcal{A}^q\times \mathcal{A}^q$ for which there exist $\textbf{u}, \textbf{v}\in\mathcal{A}^{\Z^+}$ such that $\textbf{k}\textbf{u}$ and $\textbf{l}\textbf{v}$ are $(\eps,\delta)$-tangent at $x_0$.
Let $$E(q, x_0)=\bigcap_{\eps>0}\bigcap_{\delta>0} E(q, x_0; \eps,\delta)$$ and
$$e(q,x_0)=\max_{\textbf{k}\in \mathcal{A}^q} \#\{\textbf{l}\in\mathcal{A}^q: (\textbf{k},\textbf{l})\in E(q, x_0)\}.$$

For $J\subset \R$, define $$E(q, J;\eps,\delta)=\bigcup_{x_0\in J} E(q,x_0;\eps,\delta),$$
$$E(q,J)=\bigcap_{\eps>0}\bigcap_{\delta>0} E(q, J;\eps,\delta)$$ and
$$e(q, J)=\max_{\textbf{k}\in\mathcal{A}^q} \# \{\textbf{l}\in\mathcal{A}^q: (\textbf{k},\textbf{l})\in E(q,J)\}.$$ 
Tsujii's notation $e(q)$ is defined as
$$e(q)=\lim_{p\to\infty} \max_{k=0}^{b^p-1} e\left(q, \left[\frac{k}{b^p},\frac{k+1}{b^p}\right]\right).$$
The following was proved in~\cite{Tsu}, see Proposition 8 in Section 4.
\begin{theorem}[Tsujii] \label{thm:Tsujii}
If there exists a positive integer $q$ such that $e(q)< (\gamma b)^q$, then the SRB measure $\vartheta$ is absolutely continuous with respect to the Lebesgue measure on $\R/\Z\times \R$ with square integrable density. In particular, for Lebesgue a.e. $x\in [0,1)$, $m_x$ is absolutely continuous with respect to the Lebesgue measure on $\R$ and with square integrable density.
\end{theorem}

{\em Remark.} It is obvious that $e(q)\ge e(q, x_0)$ for all $x_0\in [0,1)$. Indeed, by Porposition~\ref{prop:compact} and Lemma~\ref{lem:addingmachine}, one can prove $e(q)=\max_{x\in [0,1)} e(q,x)=\max_{x\in\R} e(q,x)$, although we do not need this fact.

We are going to define $\sigma(q)$.
Let us say that a measurable function $\omega: [0,1)\to (0,\infty)$ is a {\em weight function} if $\|\omega\|_\infty<\infty$ and $\|1/\omega\|_\infty<\infty$. A {\em testing function of order $q$} is a measurable function $V: [0,1)\times \mathcal{A}^q \times \mathcal{A}^q\to [0,\infty)$.
A testing function of order $q$ is called {\em admissible } if there exist $\eps>0$ and $\delta>0$ such that the following hold:
For any $x\in [0,1)$, if $(\textbf{u},\textbf{v})\in E(q,x; \eps, \delta)$, then
$$V(x, \textbf{u},\textbf{v}) V(x, \textbf{v}, \textbf{u})\ge 1.$$
So in particular, we have $V(x,\textbf{u},\textbf{u})\ge 1$ for each $x\in [0,1)$ and each $\textbf{u}\in\mathcal{A}^q$.

Given a weight function $\omega$ and an admissible testing function $V$ of order $q$, define a new measurable function $\Sigma_{V,\omega}^q: [0,1)\to \R$ as follows: For each $x\in [0,1)$, let
$$\Sigma^q_{V,\omega}(x)=\sup\left\{ \frac{\omega(x)}{\omega(x(\textbf{u}))}  \sum_{\textbf{v}\in\mathcal{A}^q} V(x, \textbf{u}, \textbf{v}): \textbf{u}\in\mathcal{A}^q\right\}.$$
Define $$\sigma(q)=\inf \|\Sigma^q_{V,\omega}\|_\infty,$$
where the infimum is taken over all weight functions $\omega$ and admissible testing functions $V$ of order $q$.
In \S~\ref{subsec:improveTsujii}, we shall prove the following theorem:
\begin{theorem} \label{thm:smallw2ac}
If there exists an integer $q\ge 1$ such that $\sigma (q)< (\gamma b)^q$ then the SRB measure $\vartheta$ is absolutely continuous
with respect to the Lebesgue measure on $\R/\Z\times \R$ with square integrable density. In particular, for Lebesgue a.e. $x\in [0,1)$, $m_x$ is absolutely continuous with respect to the Lebesgue measure on $\R$ and with square integrable density.
\end{theorem}

The parameter $\sigma (q)$ takes into account the fact that the number $$\#\{\textbf{v}: (\textbf{u},\textbf{v})\in E(q,x;\eps,\delta)\}$$ may depend on $x$ and $\textbf{u}$ in a significant way. On the other hand, the parameter $e(q)$ is the supremum of such  numbers over all possible choices of $x$ and $\textbf{u}$.

\begin{lemma} $\sigma(q)\le e(q)$.
\end{lemma}
\begin{proof} Fix $\eps, \delta>0$. Let $\omega=1$ be the constant weight function. For each $x\in [0,1)$, define
$$V(x,\textbf{u},\textbf{v})=\left\{
\begin{array}{ll}
1, & \mbox{ if } (\textbf{u}, \textbf{v})\in E(q, x; \eps,\delta);\\
&\\
0, &\mbox{ otherwise.}
\end{array}
\right.
$$
Then for any $x\in [0,1)$ and $\textbf{u}\in\mathcal{A}^q$, we have
$$\frac{\omega(x)}{\omega(x(\textbf{u}))}\sum_{\textbf{v}} V(x, \textbf{u}, \textbf{v})= \# \{\textbf{v}: (\textbf{u},\textbf{v})\in E(q, x;\eps,\delta)\}.$$
Thus
$$\sigma(q)\le \|\Sigma_{V,\omega}\|_\infty \le \sup_{x\in [0,1),\textbf{u}\in\mathcal{A}^q} \#\{\textbf{v}: \textbf{u},\textbf{v}\in E(q,x;\eps,\delta)\}.$$
Letting $\eps,\delta\to 0$, we obtain $\sigma(q)\le e(q)$.
\end{proof}

The following proposition collects a few facts about the quantifiers in the transversality conditions.
\begin{prop}\label{prop:compact}
For $\textbf{k},\textbf{l}\in\mathcal{A}^q$, the following hold:
\begin{enumerate}
\item For any $x_0\in\R$, $(\textbf{k},\textbf{l})\in E(q, x_0)$ if and only if there exist $\textbf{u}$ and $\textbf{v}$ in $\mathcal{A}^{\Z^+}$ such that $S(x,\textbf{ku})-S(x, \textbf{lv})$ has a multiple zero at $x_0$.
\item If $(\textbf{k},\textbf{l})\not\in E(q,x_0)$, then there is a neighborhood $U$ of $x_0$ and $\eps,\delta>0$, such that $(\textbf{k},\textbf{l})\not\in E(q,U;\eps,\delta)$.
\item For any compact $K\subset \R$, if $(\textbf{k},\textbf{l})\not\in E(q, K)$, then there exist $\eps,\delta>0$ such that $(\textbf{k},\textbf{l})\not\in E(q,K;\eps,\delta)$.
\item For any $\eps>\eps'>0,\delta>\delta'>0$ there exists $\eta>0$ such that if $|x-x_0|<\eta$, $(\textbf{k},\textbf{l})\not\in E(q,x_0;\eps,\delta)$ then
$(\textbf{k},\textbf{l})\not\in E(q,x;\eps',\delta')$.
\end{enumerate}
\end{prop}

\begin{proof}
Let us endow $\mathcal{A}^{\Z^+}$ with the usual product topology of the discrete topology on $\mathcal{A}$. Then $\mathcal{A}^{\Z^+}$ is compact. Moreover, if $\textbf{u}^n\to \textbf{u}$ in $\mathcal{A}^{\Z^+}$, then
$$S(x,\textbf{u}^n)\to S(x,\textbf{u})\text{ and } S'(x,\textbf{u}^n)\to S'(x,\textbf{u})$$
uniformly as $n\to\infty$.

(1) The ``if'' part is obvious. For the ``only if'' part, assume $(\textbf{k},\textbf{l})\in E(q,x_0)$. Then for any $n=1,2,\ldots$, $(\textbf{k},\textbf{l})\in E(q,x_0;1/n,1/n)$, and so there exist $\textbf{u}^n,\textbf{v}^n\in \mathcal{A}^{\Z^+}$ such that
$$|S(x_0,\textbf{ku}^n)-S(x_0,\textbf{lv}^n)|\le 1/n, \text{ and } |S'(x_0,\textbf{ku}^n)-S'(x_0,\textbf{lv}^n)|\le 1/n.$$
After passing to a subsequence, we may assume $\textbf{u}^n\to\textbf{u}$ and $\textbf{v}^n\to\textbf{v}$ in $\mathcal{A}^{\Z^+}$ as $n\to\infty$. Then
$$S(x_0,\textbf{ku})-S(x_0,\textbf{lv})=S'(x_0,\textbf{ku})-S'(x_0,\textbf{lv})=0.$$

(2) Arguing by contradiction, assume that the statement is false. Then there exists $\{x_n\}_{n=1}^\infty$ such that $x_n\to x_0$ and $(\textbf{k},\textbf{l})\in E(q,x_n;1/n,1/n)$. Thus there exist $\textbf{u}^n,\textbf{v}^n\in \mathcal{A}^{\Z^+}$ such that
$$|S(x_n,\textbf{ku}^n)-S(x_n,\textbf{lv}^n)|\le 1/n, \text{ and } |S'(x_n,\textbf{ku}^n)-S'(x_n,\textbf{lv}^n)|\le 1/n.$$
After passing to a subsequence we may assume $\textbf{u}^n\to\textbf{u}$, $\textbf{v}^n\to\textbf{v}$. It follows that $$S(x_0,\textbf{ku})-S(x_0,\textbf{lv})=S'(x_0,\textbf{ku})-S'(x_0,\textbf{lv})=0,$$
a contradiction.

(3) follows from (2).

(4) Since $\psi$ is $\Z$-periodic and $C^1$, for any $\xi>0$ there exists $\eta>0$ such that if $|x_1-x_2|<\eta$, then $|\psi(x_1)-\psi(x_2)|<\xi$ and $|\psi'(x_1)-\psi'(x_2)|<\xi$. Then for any $\textbf{u}\in \mathcal{A}^{\Z^+}$, we have
$$|S(x_1,\textbf{u})-S(x_2,\textbf{u})|\le \xi/(1-\gamma),$$
$$|S'(x_1,\textbf{u})-S'(x_2,\textbf{u})|\le \xi/(b-\gamma) .$$
The statement follows.
\end{proof}
We shall also use the following symmetry of the functions $S(x,\textbf{u})$.

\begin{lemma}\label{lem:addingmachine} For any $\textbf{u}\in\mathcal{A}^{\Z^+}$, $x\in \R$ and $q\in\Z^+$ ,  we have $e(q,x+1)=e(q,x)$ and $m_{x+1}=m_x$.
\end{lemma}
\begin{proof} Indeed, for any $\textbf{u}\in\mathcal{A}^{\Z^+}$ and $x\in\R$, we have
$$S(x+1, \textbf{u})=S(x, \text{add} (\textbf{u})),$$
where $\text{add}:\mathcal{A}^{\Z^+}\to \mathcal{A}^{\Z^+}$ the adding machine which can be defined as follows: Given $\textbf{u}=\{u_n\}_{n=1}^\infty\in \mathcal{A}^{\Z^+}$, defining inductively $v_n, w_n\in\mathcal{A}$ with the following properties:
\begin{itemize}
\item $w_1=1$;
\item If $u_n+w_n<b$ then $v_n=u_n+w_n$ and $w_{n+1}=0$; otherwise, define $v_n=0$ and $w_n=1$,
\end{itemize}
then $\text{add}(\textbf{u})=\{v_n\}_{n=1}^\infty$.
This is a homeomorphism of $\mathcal{A}^{\Z^+}$ which preserves the Bernoulli measure $\mathbb{P}$. Thus $m_{x+1}=m_x$.

Since the first $q$ elements of $\text{add}(\textbf{u})$ depend only on the first $q$ element of $\textbf{u}$, $\text{add}$ induces a bijection from $\mathcal{A}^q$ onto itself, denoted also  by $\text{add}$.
By definition, $(\textbf{k},\textbf{l})\in E(q, x+1)$ if and only if $(\text{add}(\textbf{k}), \text{add}(\textbf{l}))\in E(q,x)$. Thus $e(q,x+1)=e(q,x)$.
\end{proof}

\subsection{Proof of Theorem~\ref{thm:smallw2ac}}\label{subsec:improveTsujii}
The proof of Theorem~\ref{thm:smallw2ac} is an easy modification of Tsujii's proof of Theorem~\ref{thm:Tsujii}.
Fix a weight function $\omega$ and an admissible testing function $V$ of order $q$ such that
$$\|\Sigma_{V,\omega}\|_\infty < (\gamma b)^q.$$
By definition, there exist $\eps, \delta>0$ such that for any $x\in [0,1)$, if $(\textbf{u},\textbf{v})\in E(q,x;\eps,\delta)$, then
$$V(x, \textbf{u},\textbf{v})V(x,\textbf{v},\textbf{u})\ge 1.$$
 For Borel measures $\rho$ and $\rho'$ on $\R$ any $r>0$, let
$$(\rho, \rho')_r=\int_{\R} \rho(B(y,r))\rho'(B(y,r)) dy,$$
and let $$\|\rho\|_r=\sqrt{(\rho, \rho)_r}.$$
For a Borel subset $J\subset \R$, define
$$I_r (J)=\frac{1}{r^2}\int_{J} \omega(x) \|m_x\|_r^2 dx$$ and write $I_r=I_r([0,1))$.

According to Lemma 4 of \cite{Tsu}, $\liminf_{r\to 0}\|\rho\|_r<\infty$ implies that $\rho$ is absolutely continuous with respect to the Lebesgue measure and the density function is square integrable. Consequently, if $\liminf I_r <\infty$, then the conclusion of the theorem holds.

With slight abuse of language, let $T^q(m_x)$ denote the pushforward of the measure $m_x$ under the map $y\mapsto \pi_2\circ T^q(x,y)$, where $\pi_2(x,y)=y$. Then
$$m_x=\frac{1}{b^q}\sum_{\textbf{i}\in\mathcal{A}^q} T^q(m_{x(\textbf{i})}).$$
Thus
$$\|m_x\|_r^2=b^{-2q}\sum_{\textbf{i},\textbf{j}} (T^q m_{x(\textbf{i})}, T^q m_{x(\textbf{j})})_r.$$
Let
$$I_r^o (J)=\frac{1}{b^{2q}r^2} \int_J \omega(x)\sum_{(\textbf{i},\textbf{j})\not\in E(q, x;\eps,\delta)} (T^q(m_{x(\textbf{i})}),T^q (m_{x(\textbf{j})}))_r dx$$
and
$$I_r^* (J)=\frac{1}{b^{2q}r^2} \int_J \omega(x)\sum_{(\textbf{i},\textbf{j})\in E(q, x;\eps,\delta)} (T^q(m_{x(\textbf{i})}),T^q (m_{x(\textbf{j})}))_r dx.$$
Then
$$I_r(J)=I_r^o(J) +I_r^* (J).$$
We shall also write $I_r^o=I_r^0([0,1))$ and $I_r^*=I_r^*([0,1))$.


\begin{lemma}\label{lem:tranest}There exists $C>0$ such that $I_r^o\le C$ holds for all $r>0$.
\end{lemma}
\begin{proof} Fix $\eps'\in (0,\eps)$ and $\delta'\in (0,\delta)$. By Proposition~\ref{prop:compact} (4), there exists a positive integer $p$, such that if $x\in J_{p,k}:=[k/b^{p}, (k+1)/b^p]$ and $(\textbf{i},\textbf{j})\not\in E(q, x;\eps,\delta)$, then $(\textbf{i},\textbf{j})\not\in E(q, J_{p,k};\eps',\delta')$. It follows that
$$I_r^o\le \frac{\|\omega\|_\infty}{b^{2q}r^2}\sum_{k=0}^{b^p-1}\int_{J_{k,p}} \sum_{(\textbf{i},\textbf{j})\not\in E(q,J_{p,k};\eps',\delta')} (T^q(m_{x(\textbf{i})}),T^q (m_{x(\textbf{j})}))_r dx.$$
In Proposition 6 in~\cite{Tsu}, it was proved that there exists $C'=C'(p,\eps',\delta')>0$ such that if $(\textbf{i},\textbf{j})\not\in E(q, J_{p,k};\eps',\delta')$, then
$$\int_{J_{p,k}} (T^q(m_{x(\textbf{i})}),T^q (m_{x(\textbf{j})}))_r dx\le C'r^2.$$
Thus $I_r^o\le C.$
\end{proof}
In order to estimate the terms $I_r^*(J)$, Tsujii observed

\begin{lemma}\label{lem:Tsujiiobservation}
For any $\textbf{i}\in\mathcal{A}^q$ and any $x\in\R$, we have
$$\|T^q(m_{x(\textbf{i})})\|_r^2=\gamma^q \|m_{x(\textbf{i})}\|_{\gamma^{-q}r}^2.$$
\end{lemma}
\begin{proof} This follows immediately from the fact that 
$T^q$ is a contraction of rate $\gamma^q$ in the vertical direction.
\end{proof}

\begin{lemma}\label{lem:parallelest} For each $r>0$, we have
$$I_r^* \le \frac{\|\Sigma_{V,\omega}\|_\infty}{(b\gamma)^q}  I_{\gamma^{-q}r}.$$
\end{lemma}
\begin{proof}

Let us first prove that for each $x\in [0,1)$,
\begin{equation}\label{eqn:paratrans}
\sum_{(\textbf{u},\textbf{v})\in E(q,x;\eps, \delta)} (T^q m_{x(\textbf{u})}, T^q m_{x(\textbf{v})})_r\le \gamma^q \sum_{\textbf{u}\in\mathcal{A}^q} \left(\sum_{\textbf{v}\in\mathcal{A}^q} V(x,\textbf{u},\textbf{v})\right)\|m_{x(\textbf{u})}\|_{r/\gamma^q}^2.
\end{equation}

To this end, let $\textbf{u}_k$, $k=1,2,\ldots, b^q$ be all the elements of $\mathcal{A}^q$. Fix $x\in [0,1)$ and prepare the following notation:
$V_{kl}=V(x, \textbf{u}_k,\textbf{u}_l)$, $x_k=x(\textbf{u}_k)$ and
$$\theta_{kl}=
\left\{
\begin{array}{ll}
1 & \mbox{ if }(\textbf{u}_k,\textbf{u}_l)\in E(q,x;\eps,\delta)\\
&\\
0 & \mbox{ otherwise.}
\end{array}
\right.$$
Then
$$\sum_{(\textbf{u},\textbf{v})\in E(q,x;\eps, \delta)} (T^q m_{x(\textbf{u})}, T^q m_{x(\textbf{v})})_r=\sum_{k=1}^{b^q} \|T^q m_{x_k}\|_r^2 + 2\sum_{1\le k<l \le b^q} \theta_{kl} (T^q m_{x_k}, T^q m_{x_l})_r.$$
For each $1\le k< l\le b^q$, by the Cauchy-Schwarz inequality,
$$ (T^q m_{x_k}, T^q m_{x_l})_r \le  \|T^q m_{x_k}\|_r \|T^q m_{x_l}\|_r.$$
Thus
$$2\theta_{kl} (T^q m_{x_k}, T^q m_{x_l})_r \le V_{kl} \|T^q m_{x_k}\|_r^2 + V_{lk} \|T^q m_{x_l}\|_r^2.$$
Indeed, this is trivial if $\theta_{kl}=0$, while if $\theta_{kl}=1$, it follows from the previous inequality and $V_{kl}V_{lk}\ge 1$.
Consequently,
\begin{align*}
2\sum_{1\le k<l\le b^q } \theta_{kl} (T^q m_{x_k}, T^q m_{x_l})_r& \le \sum_{1\le k<l\le b^q}\left( V_{kl} \|T_q m_{x_k}\|_r^2 + V_{lk} \|T^q m_{x_l}\|_r^2\right)\\
& =\sum_{k=1}^{b^q} \left(\sum_{\substack{1\le l\le b^q\\ l\not=k}} V_{kl}\right) \|T^q m_{x_k}\|_r^2,
\end{align*}
and hence
$$\sum_{(\textbf{u},\textbf{v})\in E(q,x;\eps, \delta)} (T^q m_{x(\textbf{u})}, T^q m_{x(\textbf{v})})_r\le \sum_{k=1}^{b^q} \left(\sum_{l=1}^{b^q} V_{kl} \right) \|T^q m_{x_k}\|_r^2.$$
By Lemma~\ref{lem:Tsujiiobservation}, the inequality (\ref{eqn:paratrans}) follows.

Multiplying $\omega(x)$ on both sides of (\ref{eqn:paratrans}), we obtain
$$\omega(x) \sum_{(\textbf{u},\textbf{v})\in E(q,x;\eps, \delta)} (T^q m_{x(\textbf{u})}, T^q m_{x(\textbf{v})})_r
\le \gamma^q \sum_{\textbf{u}\in\mathcal{A}^q} \Sigma_{V,\omega}(x)\omega(x(\textbf{u}))\|m_{x(\textbf{u})}\|_{\gamma^{-q}r}^2.$$
Dividing both side by $b^{2q}r^2$ and integrating over $[0,1)$, we obtain
$$I_r^*\le \frac{\|\Sigma_{V,\omega}\|_\infty} {b^{2q}\gamma^{q}}\frac{1}{(\gamma^{-q}r)^2}
\sum_{\textbf{u}\in \mathcal{A}^q} \int_0^1 \|m_{x(\textbf{u})}\|_{\gamma^{-q}r}^2 \omega(x(\textbf{u}))dx.$$
Let $J(\textbf{u})=\{x(\textbf{u}): 0\le x<1\}$. Then
$$\int_0^1 \|m_{x(\textbf{u})}\|_{\gamma^{-q}r}^2\omega(x(\textbf{u})) dx =b^q\int_{J(\textbf{u})} \|m_x\|_{\gamma^{-q}r}^2\omega(x) dx.$$
Since $J(\textbf{u})$, $\textbf{u}\in\mathcal{A}^q$, form a partition of $[0,1)$, it follows that
$$I_r^*\le \frac{\|\Sigma _{V,\omega}\|_\infty}{(\gamma b)^q} \frac{1}{(\gamma^{-q}r)^2} \int_0^1 \|m_{x}\|_{\gamma^{-q}r}^2 \omega(x) dx =\frac{\|\Sigma _{V,\omega}\|_\infty}{(\gamma b)^q} I_{\gamma^{-q} r}.$$
\end{proof}

\begin{proof}[Completion of proof of Theorem~\ref{thm:smallw2ac}]
By Lemma~\ref{lem:tranest} and Lemma~\ref{lem:parallelest}, there exists a constant $C>0$ such that
$$I_r =I_r^o+I_r^*\le C+ \beta I_{\gamma^{-q} r},$$
holds for all $r>0$, where $\beta= \|\Sigma_{V,\omega}\|_\infty/ (\gamma b)^q \in (0,1).$ As $I_r<\infty$ for each $r>0$, it follows that
$\liminf_{r\searrow 0}I_r<\infty.$ By the remarks at the beginning of this subsection, the conclusion of the theorem follows. 
\end{proof}

\subsection{Proof of  Theorem~\ref{thm:gammalarge}}
In this subsection, we shall prove Theorem~\ref{thm:gammalarge} using Theorem~\ref{thm:smallw2ac}.

\begin{lemma} \label{lem:1miss}
Suppose that for each $x\in [0,1)$, $E(q,x)\not=\mathcal{A}^q \times \mathcal{A}^q$. Then
$$\sigma(q)\le b^q-2+2/\alpha,$$
where $\alpha=\alpha (b,q)>1$ satisfies
$$2-\alpha= (b^q-2) \alpha (\alpha-1).$$
\end{lemma}
\begin{proof} By Lemma~\ref{lem:addingmachine}, the assumption implies that for each $x\in\R$, $E(q,x)\not=\mathcal{A}^q\times \mathcal{A}^q$. By Proposition~\ref{prop:compact} (2) and compactness of $[0,1]$,
there exists $\eps>0, \delta>0$ such that $E(q,x;\eps,\delta)\not=\mathcal{A}^q\times \mathcal{A}^q$ for each $x\in [0,1]$. So we can find measurable functions
$\textbf{k},\textbf{l}:[0,1)\to \mathcal{A}^q$ such that $(\textbf{k}(x),\textbf{l}(x))\not\in E(q,x;\eps,\delta)$. Define $\omega(x)=1$ for all $x\in [0,1)$. Define
$$V(x,\textbf{u},\textbf{v})=\left\{
\begin{array}{ll}
1 & \text{ if } \textbf{u},\textbf{v}\not\in \{\textbf{k}(x),\textbf{l}(x)\} \text{ or } \textbf{u}=\textbf{v};\\
0 & \text{ if } (\textbf{u},\textbf{v})=(\textbf{k}(x),\textbf{l}(x)) \text { or } (\textbf{l}(x),\textbf{k}(x));\\
\alpha &\text{ if } \textbf{u}\in \{\textbf{k}(x),\textbf{l}(x)\} \text { but } \textbf{v}\not\in \{\textbf{k}(x),\textbf{l}(x)\};\\
\alpha^{-1} & \text { if }  \textbf{u}\not\in \{\textbf{k}(x),\textbf{l}(x)\} \text { but } \textbf{v}\in \{\textbf{k}(x),\textbf{l}(x)\}.
\end{array}
\right.
$$
Then $V$ is an admissible test function of order $q$. For every $x\in [0,1)$, if $\textbf{u}\not\in \{\textbf{k}(x), \textbf{l}(x)\}$, then
$$\frac{\omega(x)}{\omega(x(\textbf{u}))}\sum_{\textbf{v}\in\mathcal{A}^q} V(x, \textbf{u},\textbf{v})=b^q-2+\frac{2}{\alpha},$$
and if $\textbf{u}\in \{\textbf{k}(x), \textbf{l}(x)\}$, then
$$\frac{\omega(x)}{\omega(x(\textbf{u}))}\sum_{\textbf{v}\in\mathcal{A}^q} V(x, \textbf{u},\textbf{v})=1+(b^q-2)\alpha=b^q-2+\frac{2}{\alpha}.$$
Thus $\sigma(q)\le \|\Sigma_{V,\omega}\|_\infty \le b^q -2+2/\alpha.$
\end{proof}

We shall use some results obtained in \cite{BKRU}. Fix an integer $b\ge 2$. We say that a $\Z$-periodic continuous function $\psi: \R\to \R$ is {\em cohomologous to $0$} if there exists a continuous $\Z$-periodic function $f:\R\to \R$ such that
$\psi(x)=f(bx)-f(x)$ holds for all $x\in \R$.
The main step is the following lemma.

\begin{lemma}\label{lem:check1miss}
Assume that $\psi:\R\to \R$ is a $\Z$-periodic $C^1$ function that is not cohomologous to zero and $\int_0^1 \psi(x) dx=0$. Then there exists $\gamma_1\in (0,1)$ and a positive integer $N$ such that if $\gamma_1<\gamma<1$, then $E(N, x)\not=\mathcal{A}^N\times \mathcal{A}^N$ for each $x\in \R$.
\end{lemma}
\begin{proof} We shall prove that there exists $\gamma_1$, $N_1$ and $x_1\in \R$ such that $E(N_1, x_1)\not=\mathcal{A}^{N_1}\times\mathcal{A}^{N_1}$. Note that this is enough for the conclusion of this lemma. Indeed, let $\textbf{k},\textbf{l}\in\mathcal{A}^{N_1}$ be such that $(\textbf{k},\textbf{l})\not\in E(N_1, x_1)$. Then by Proposition~\ref{prop:compact} (2), there exist $\eps>0$ and $\delta>0$
a neighborhood $U$ of $x_1$ such that $(\textbf{k},\textbf{l})\not\in E(N_1, x_1;\eps,\delta)$.
Let $N_2$ be a positive integer such that $b^{N_2}U+\Z=\R$. Then for any $y\in\R$, there exists $x\in U$ and $k\in \Z$ such that $y=b^{N_2}x+k$. Since the words $(00\cdots 0\textbf{k}), (00\cdots0\textbf{l})\in\mathcal{A}^{N_2+N_1}$ are transversal at $b^{N_2}x$, we have
$E(N_1+N_2, y-k)\not=\mathcal{A}^{N_1+N_2}\times \mathcal{A}^{N_1+N_2}$ which is equivalent to $E(N_1+N_2, y)\not=\mathcal{A}^{N_1+N_2}\times \mathcal{A}^{N_1+N_2}$ by Lemma~\ref{lem:addingmachine}.

For each $\textbf{u}\in\mathcal{A}^{\Z^+}$, let $G:\R\to \R$ be defined as
$$G(x,\textbf{u})=\sum_{n=1}^\infty \frac{1}{b^{n}} \psi'\left(\frac{x+u_1+u_2 b+\cdots +u_n b^{n-1}}{b^n}\right).$$

Note that $G(x)=G(x,\textbf{0})$ satisfies the functional equation
$$bG(bx)= \psi'(x)+ G(x).$$

We claim that $G(x)$ is not $\Z$-periodic. Indeed, otherwise, from the equation above, we obtain $\int_0^1 G(x) dx=0$. Then $g(x)=\int_0^x G(t)dt$ defines a $\Z$-periodic function.
Since $\psi'= b g'(bx)-g'(x)$ and
$$\int_0^1 \psi(x) dx=\int_0^1 (g(bx)-g(x)) dx=0,$$
it follows that $\psi(x)= g(bx)-g(x)$ holds for all $x$. This contradicts the assumption that $\psi$ is not cohomologous to zero.

Since $G(x+1)=G(x, (100\cdots))$, it follows that there exists $x_1\in [0,1)$ such that
$$5\delta:=|G(x_1, (000\cdots))-G(x_1, (100\cdots))|>0.$$
Let $C=\max_{x\in [0,1]} |\psi'(x)|$. Let $N_1$ be a positive integer such that $$2C<\delta b^{N_1}(b-1)$$ and let
 $\gamma_1\in (0,1)$ be such that
 $$(1-\gamma_1^{N_1})C<\delta (b-1).$$
Then, for any $\textbf{k}, \textbf{l}\in\mathcal{A}^{\Z^+}$ with $k_1=k_2=\cdots=k_{N_1}=0$, $l_1=1$, $l_2=l_3=\cdots=l_{N_1}=0$, we have
\begin{multline*}
\left|\frac{d}{dx} S(x_1,\textbf{k})-G(x_1)\right|\le \sum_{n=1}^{N_1} \frac{(1-\gamma ^{n-1})}{b^n} C + 2C \sum_{n=N_1+1}^\infty b^{-n}\\
< (1-\gamma_1^{N_1}) C(b-1)^{-1}+ 2C((b-1) b^{N_1})^{-1}< 2\delta,
\end{multline*}
and similarly,
$$
\left|\frac{d}{dx} S(x_1,\textbf{l})-G(x_1+1)\right|< 2\delta.
$$
Therefore,
$$\left|\frac{d}{dx} S(x_1,\textbf{k})-\frac{d}{dx} S(x_1,\textbf{l})\right|\ge \delta.$$
It follows that the two words $(00\cdots 0), (10\cdots 0)\in \mathcal{A}^{N_1}$ are transversal at $x_1$, hence
$$E(N_1, x_1)\not=\mathcal{A}^{N_1}\times \mathcal{A}^{N_1}.$$
\end{proof}

\begin{proof}[Proof of Theorem~\ref{thm:gammalarge}]
Let $\psi=\phi'$, so $\psi$ is a $\Z$-periodic non-constant $C^1$ function and $\int_0^1 \psi(x) dx=0$.  Consider the map $T$ as in (\ref{eqn:mapT}). By Ledrappier's Theorem, it suffices to prove the measures $m_x$ defined for this map $T$ are absolutely continuous for Lebesgue almost every $x\in [0,1]$.

First we assume that $\psi$ is not cohomologous to $0$. By Lemma~\ref{lem:check1miss}, there exists $\gamma_1\in (0,1)$ and $N$ such that if $\gamma_1<\gamma<1$, then $E(N, x)\not=\mathcal{A}^N\times \mathcal{A}^N$ for each $x\in \R$.
By Lemma~\ref{lem:1miss}, this implies that $\sigma(N)< b^N-2+\frac{2}{\alpha}$ where $\alpha=\alpha (b,N)\in (1,2)$. Thus there exists $\gamma_0\in (\gamma_1, 1)$ such that if $\gamma>\gamma_0$ then $\sigma(N)< (b\gamma)^N$. By Theorem~\ref{thm:smallw2ac}, it follows that $m_x$ is absolutely continuous for Lebesgue a.e. $x\in [0,1]$.

To complete the proof, we shall use a few results of ~\cite{BKRU}.  Assume that $\psi$ is cohomologous to $0$ and let $\psi_1:\R\to \R$ be a $\Z$-periodic continuous function such that $\psi_1(bx)-\psi_1(x)=\psi(x)$. By Lemma 5.2 (5) and Lemma 5.8 (2) of that paper, $\psi_1$ is $C^1$, and $T_{b,\gamma, \psi}$ is $C^1$ conjugate to $T_{b,\gamma, \psi_1}$. By adding a constant if necessary, we may assume $\int_0^1 \psi_1(x) dx =0$. If $\psi_1$ is not cohomologous to zero, then we are done. Otherwise, repeat the argument. By Lemma 5.6 of that paper, any $\Z$-periodic non-constant $C^1$ function $\psi$ is not infinitely cohomologous to zero. Thus the procedure stops within finitely many steps.
\end{proof}

\subsection{Plan of Proof of Theorem~\ref{thm:cossole}}
Theorem~\ref{thm:cossole} follows from the following theorem by Theorem~\ref{thm:smallw2ac}. 
\begin{theorem}\label{thm:reduced}
For an integer  $b\ge 2$, $1/b<\gamma<1$ and $\psi(x)=-2\pi \sin (2\pi x)$, consider the map $T$ as in (\ref{eqn:mapT}). Then there exists a positive integer $q$ such that
$\sigma(q)<(b\gamma)^q.$
\end{theorem}
The rest of the paper is devoted to the proof of Theorem~\ref{thm:reduced}. The proof uses special property of the map $\psi$ and breaks into several cases.
\begin{proof}[Proof of Theorem~\ref{thm:reduced}]
The case $b\ge 6$ is proved in Theorem~\ref{thm:blarge} (i). The case $b=5$ is proved in Theorem~\ref{thm:b=5}. The case $b=4$ is proved in Theorem~\ref{thm:b=4}. The case $b=3$ is proved in Theorem~\ref{thm:b=3}. The case $b=2$ follows from Corollary~\ref{cor:b=2.4} and Proposition~\ref{prop:b=2step5}.
\end{proof}

To conclude this section, we include a few lemmas which will be used
in later sections. The first lemma is about a new symmetric property of the functions $S(x, \textbf{u})$ in the case that $\psi$ is odd.

\begin{lemma}[Symmetry]\label{lem:symmetric} Assume that $\psi(x)$ is an odd function. Then for any $\textbf{i}=\{i_n\}_{n=1}^\infty\in\mathcal{A}^{\Z^+}$, letting $\textbf{i}'=\{i_n'\}_{n=1}^\infty$ with $i_n'=b-1-i_n$, we have
$$-S(x,\textbf{i})=S(1-x, \textbf{i}').$$
\end{lemma}
\begin{proof} This follows from the definition of $S(\cdot, \cdot)$.
\end{proof}
The next three lemmas will be used to obtain upper bounds for $\sigma(q)$.
\begin{lemma}\label{lem:sqrt2}
Let $q\ge 1$ be an integer. Suppose that there are constants $\eps>0$ and $\delta>0$ and $K\subset [0,1)$ with the following properties:
\begin{enumerate}
\item [(i)] For $x\in K$,  $e(q,x;\eps,\delta)=1$ and for $x\in [0,1)\setminus K$, $e(q,x;\eps,\delta)\le 2$;
\item [(ii)] If $(\textbf{u},\textbf{v})\in E(q,x;\eps,\delta)$ for some $x\in [0,1)\setminus K$ and $\textbf{u}\not=\textbf{v}$, then both
$x(\textbf{u})$ and  $x(\textbf{v})$ belong to $K$.
\end{enumerate}
Then $\sigma(q)\le \sqrt{2}$.
\end{lemma}

\begin{proof} We define suitable weight function $\omega$ and testing function $V$. Let $L=[0,1)\setminus K$. 
Define
$$\omega(x)=\left\{
\begin{array}{ll}
\sqrt{2} & \mbox{ if } x\in K;\\
1 &\mbox{ if } x\in L.
\end{array}
\right.
$$
Define
$$V(x,\textbf{u}, \textbf{v})=\left\{
\begin{array}{ll}
1 &\mbox{ if } (\textbf{u}, \textbf{v})\in E(q,x;\eps,\delta);\\
0 &\mbox{ otherwise.}
\end{array}
\right.
$$
Then for $x\in K$ and any $\textbf{u}\in\mathcal{A}^q$,
$$\frac{\omega(x)}{\omega(x(\textbf{u}))} \sum_{\textbf{v}} V(x, \textbf{u}, \textbf{v})=\frac{\omega(x)}{\omega(x(\textbf{u}))}\le \sqrt{2}.$$
For $x\in L$ and $\textbf{u}\in \mathcal{A}^q$, if $\textbf{u}$ is not $(\eps,\delta)$-tangent to any other element of $\mathcal{A}^q$ at $x$, then we have
$$\frac{\omega(x)}{\omega(x(\textbf{u}))} \sum_{\textbf{v}} V(x, \textbf{u}, \textbf{v})=\frac{\omega(x)}{\omega(x(\textbf{u}))}\le 1;$$
otherwise, we have $\omega(x)=1$ and $\omega(x(\textbf{u}))=\sqrt{2}$
$$\frac{\omega(x)}{\omega(x(\textbf{u}))} \sum_{\textbf{v}} V(x, \textbf{u}, \textbf{v})=\frac{\omega(x)}{\omega(x(\textbf{u}))}\cdot 2\le \sqrt{2}.$$
It follows that $\sigma(q)\le \Sigma_{V,w}\le \sqrt{2}.$
\end{proof}

\begin{lemma}\label{lem:goldenratio}
Let $q\ge 1$ be an integer. Suppose that there are constants $\eps>0$ and $\delta>0$ and $K\subset [0,1)$ with the following properties:
\begin{enumerate}
\item [(i)] For $x\in K$, $e(q,x;\eps,\delta)\le 1$ and for $x\in [0,1)\setminus K$, $e(q, x:\eps,\delta)\le 2$;
\item [(ii)] If $(\textbf{u},\textbf{v})\in E(q,x;\eps,\delta)$ for some $x\in [0,1)\setminus K$ and $\textbf{u}\not=\textbf{v}$, then either $x(\textbf{u})\in K$ or
$x(\textbf{v})\in K.$
\end{enumerate}
Then $\sigma(q)\le (\sqrt{5}+1)/2$.
\end{lemma}
\begin{proof} Let $L=[0,1)\setminus K$. Define 
$$\omega(x)=\left\{
\begin{array}{ll}
(\sqrt{5}+1)/2 &\mbox{ if } x\in K;\\
1 &\mbox{ otherwise.}
\end{array}
\right.$$
For $x\in K$, define
$$V(x, \textbf{u}, \textbf{v})=\left\{
\begin{array}{ll}
1 & \mbox{ if } \textbf{u}=\textbf{v}\\
0 &\mbox{ otherwise.}
\end{array}
\right.
$$
%
For $x\in L$, define
$$V(x,\textbf{u}, \textbf{v})=\left\{
\begin{array}{ll}
0 & \mbox{ if } (\textbf{u}, \textbf{v})\not\in E(q, x;\eps,\delta);\\
1 & \mbox{ if } \textbf{u}=\textbf{v};\\
(\sqrt{5}+1)/2 & \mbox{ if } (\textbf{u},\textbf{v})\in E(q, x;\eps,\delta), \textbf{u}\not=\textbf{v}, \text{ and } x(\textbf{u})\in K;\\
(\sqrt{5}-1)/2 & \mbox{ if } (\textbf{u},\textbf{v})\in E(q, x;\eps,\delta), \textbf{u}\not=\textbf{v}, \text{ and } x(\textbf{u})\not\in K.
\end{array}
\right.
$$
Then for $x\in K$ and any $\textbf{u}\in\mathcal{A}^q$, we have
$$\frac{\omega(x)}{\omega(x(\textbf{u}))}\sum_{\textbf{v}} V(x,\textbf{u},\textbf{v})\le \frac{\sqrt{5}+1}{2}.$$
For $x\in L$ and $\textbf{u}\in\mathcal{A}^q$, if $x(\textbf{u})\in K$, we have 
$$\frac{\omega(x)}{\omega(x(\textbf{u}))}\sum_{\textbf{v}} V(x,\textbf{u},\textbf{v})\le \frac{1}{(\sqrt{5}+1)/2} \left(\frac{\sqrt{5}+1}{2}+1\right)=\frac{\sqrt{5}+1}{2};$$
if $x(\textbf{u})\not \in K$, then
$$\frac{\omega(x)}{\omega(x(\textbf{u}))}\sum_{\textbf{v}} V(x,\textbf{u},\textbf{v})\le \frac{1}{1} \left(\frac{\sqrt{5}-1}{2}+1\right)=\frac{\sqrt{5}+1}{2}.$$
In conclusion, we have $\sigma(q)\le \|\Sigma_{V,\omega}\|_\infty \le (\sqrt{5}+1)/2.$
\end{proof}
The next lemma is more technical and will only be used in the case $b=2$.
\begin{lemma}\label{lem:w1.7}
Let $q\ge 1$ be an integer. Suppose there are three pairwise disjoint subsets $K_0, K_1, K_2$ of $[0,1)$ with $K_0\cup K_1\cup K_2 =[0,1)$ and constants $\eps,\delta>0$ such that the following hold:
\begin{enumerate}
\item [(i)] For each $x\in K_0$, $e(q, x;\eps,\delta)=1$;
\item [(ii)] For each $x\in K_1$, there exist $\textbf{a}_x,\textbf{b}_x\in \mathcal{A}^q$ such that $x(\textbf{a}_x), x(\textbf{b}_x)\in K_0$ and such that $(\textbf{a}_x, \textbf{b}_x)$ and $(\textbf{b}_x,\textbf{a}_x)$ are the only possible non-trivial element of $E(q, x;\eps, \delta)$;
\item [(iii)] For $x\in K_2$, there exist $\textbf{a}_x, \textbf{b}_x,\textbf{c}_x\in\mathcal{A}^q$ such that $x(\textbf{a}_x), x(\textbf{b}_x)\in K_0$ and $x(\textbf{c}_x)\in K_1$ and such that $(\textbf{a}_x,\textbf{b}_x)$, $(\textbf{a}_x, \textbf{c}_x)$, $(\textbf{b}_x,\textbf{a}_x)$
    and $(\textbf{c}_x,\textbf{a}_x)$ are the only possible non-trivial elements of $E(q, x;\eps,\delta)$.
\end{enumerate}
Then $$\sigma(q)\le t<1.61,$$
where $t>\sqrt{2}$ is the unique solution of the following equation
\begin{equation}\label{eqn:w3t}
\frac{1}{t^2-1}+\frac{2}{t^3-2}+1=t^2.
\end{equation}
\end{lemma}
\begin{proof}  Let $s=t^2/2$. Note that $t>s>1$. Define
$$\omega(x)=\left\{
\begin{array}{ll}
t & \mbox{ if } x\in K_0;\\
s & \mbox{ if } x\in K_1;\\
1 & \mbox{ if } x\in K_2.
\end{array}
\right.
$$

For $x\in K_0\cup K_1$, define
$$V(x, \textbf{u}, \textbf{v})=\left\{
\begin{array}{ll}
0 & \mbox{ if } (\textbf{u},\textbf{v})\not\in E(q, x;\eps,\delta);\\
1 & \mbox{ otherwise.}
\end{array}
\right.
$$
For $x\in K_2$, define
$$V(x,\textbf{u}, \textbf{v})=\left\{
\begin{array}{ll}
1 & \mbox{ if } \textbf{u}=\textbf{v};\\
t s-1 & \mbox{ if } (\textbf{u},\textbf{v})=(\textbf{c}_x,\textbf{a}_x);\\
(t s-1)^{-1} & \mbox{ if } (\textbf{u},\textbf{v})=(\textbf{a}_x,\textbf{c}_x);\\
  t^2-1 & \mbox{ if } (\textbf{u},\textbf{v})=(\textbf{b}_x,\textbf{a}_x);\\
(t^2-1)^{-1} & \mbox{ if } (\textbf{u},\textbf{v})=(\textbf{a}_x,\textbf{b}_x)\\
0 &\mbox{ otherwise.}
\end{array}
\right.
$$
Then for $x\in K_0$, and any $\textbf{u}\in\mathcal{A}^q$ we have
$$\frac{\omega(x)}{\omega(x(\textbf{u}))}\sum_{\textbf{v}\in\mathcal{A}^q} V(x, \textbf{u}, \textbf{v})\le t;$$
for $x\in K_1$, $\textbf{u}\in\{\textbf{a}_x, \textbf{b}_x\}$,
$$\frac{\omega(x)}{\omega(x(\textbf{u}))}\sum_{\textbf{v}\in\mathcal{A}^q} V(x, \textbf{u}, \textbf{v})= 2s/t=t;$$
for $x\in K_1$, $\textbf{u}\not\in\{\textbf{a}_x, \textbf{b}_x\}$,
$$\frac{\omega(x)}{\omega(x(\textbf{u}))}\sum_{\textbf{v}\in\mathcal{A}^q} V(x, \textbf{u}, \textbf{v})\le s<t;$$
for $x\in K_2$,
$$\frac{\omega(x)}{\omega(x(\textbf{a}_x))}\sum_{\textbf{v}\in\mathcal{A}^q} V(x, \textbf{u}, \textbf{v})= \frac{1}{t} \left(\frac{1}{t^2-1}+\frac{1} {t s-1}+1\right)=t;$$
for $x\in K_2$, 
$$\frac{\omega(x)}{\omega(x(\textbf{b}_x))}\sum_{\textbf{v}\in\mathcal{A}^q} V(x, \textbf{b}_x, \textbf{v})=\frac{1}{t}\left(1+t^2-1\right)= t;$$
for $x\in K_2$, 
$$\frac{\omega(x)}{\omega(x(\textbf{c}_x))}\sum_{\textbf{v}\in\mathcal{A}^q} V(x, \textbf{c}_x, \textbf{v})=\frac{1}{s}\left(1+ts-1\right)= t;$$
and for $x\in K_2,$, $\textbf{u}\not\in\{\textbf{a}_x,\textbf{b}_x,\textbf{c}_x\}$,
$$\frac{\omega(x)}{\omega(x(\textbf{u}))}\sum_{\textbf{v}\in\mathcal{A}^q} V(x, \textbf{u}, \textbf{v})\le 1.$$
Therefore
$$\sigma(q)\le \|\Sigma_{V,\omega}(x)\|_\infty\le t.$$
\end{proof}

\section{The case when $b$ is large}
In this and next sections, we shall prove Theorem~\ref{thm:reduced}. So we consider a map $T$ of the form (\ref{eqn:mapT}) with $\psi=-2\pi\sin (2\pi x).$ The main result of this section is the following:
\begin{theorem}\label{thm:blarge}
\begin{enumerate}
\item If $b\ge 6$, then $\sigma(1)\le e(1) <\gamma b$.
\item If $b=4,5$, then either $e(1)= 2$ or $e(1)<\gamma b$.
\item If $b=3$, then either $e(1)=2$ or $\sigma(1)<\gamma b$.
\end{enumerate}
\end{theorem}

We start with a few lemmas.
Let $$\Delta_{b, \gamma}=\max_{t\in \R} (\sin (bt)+\gamma \sin (t)).$$
Besides the trivial bound: $\Delta_{b,\gamma}\le 1+\gamma$, we also need the following:
\begin{lemma}\label{lem:Delta} For each $\gamma\in (0,1)$, we have
\begin{align}\label{eqn:Delta6}
\Delta_{6,\gamma}& \le \max(1+0.972\gamma, 0.99+\gamma)\\
\label{eqn:Delta3}
\Delta_{3,\gamma} & \le 1+0.71\gamma.
\end{align}
\end{lemma}
\begin{proof}
Let us first prove (\ref{eqn:Delta3}). Indeed, if $\sin t\le 0.71$ then the inequality holds. So assume $\sin t>0.71$. Then $\sin (3t)=3\sin t -4\sin^3 t\le 0.71$, and hence $\sin (3t) +\gamma \sin t\le 0.71+ \gamma \le 1+0.71\gamma$.

Let us prove (\ref{eqn:Delta6}). If $\sin t \le 0.972$ then the inequality holds. So assume $\sin t> 0.972$, then
$\sin (3t) \le 3\cdot 0.972- 4\cdot 0.972^3=-0.757$. Then
$$|\sin (6t)|\le 2|\sin (3t)|\sqrt{1-\sin^2(3t)}<0.99.$$
\end{proof}
\begin{lemma}\label{lem:1stest}
If $(k,l)\in E(1, x^*)$, then
\begin{equation}\label{eqn:c0bd}
\left|\sin \frac{2\pi(x^*+k)}{b}-\sin \frac{2\pi(x^*+l)}{b}\right|\le \frac{2\Delta_{b,\gamma} \gamma} {1-\gamma^2}\le \frac{2\gamma}{1-\gamma},
\end{equation}
\begin{equation}\label{eqn:c1bd}
\left|\cos \frac{2\pi(x^*+k)}{b}-\cos \frac{2\pi(x^*+l)}{b}\right|\le \frac{2\gamma}{b-\gamma},
\end{equation}
\begin{equation}\label{eqn:sumC0C1}
4\sin^2 \frac{\pi (k-l)}{b}\le \left(\frac{2\gamma\Delta_{b,\gamma}}{1-\gamma^2}\right)^2+\left(\frac{2\gamma}{b-\gamma}\right)^2\le
\left(\frac{2\gamma}{1-\gamma}\right)^2+\left(\frac{2\gamma}{b-\gamma}\right)^2.
\end{equation}
\end{lemma}
\begin{proof} By Proposition~\ref{prop:compact} (1), there exists $\textbf{k}=\{k_n\}_{n=1}^\infty$ and $\textbf{l}=\{l_n\}_{n=1}^\infty$ in $\mathcal{A}^{\Z^+}$ with $k_1=k$ and $l_1=l$ such that for $F(x):=-(2\pi)^{-1}(S(x,\textbf{k})-S(x,\textbf{l}))$, we have $F(x^*)=F'(x^*)=0$.
Let $f(x)=\sin (2\pi bx)+\gamma \sin (2\pi x)$.
Then
$$\left|-\frac{S(x,\textbf{k})}{2\pi}-\sin \frac{2\pi(x+k)}{b}\right|\le \sum_{n=1}^\infty \gamma^{2n-1} |f (x_{2n+1})|\le \frac{2\gamma \Delta_{b,\gamma}}{1-\gamma^2},$$
 and
$$\left|-\frac{bS'(x,\textbf{k})}{4\pi^2}-\cos \frac{2\pi(x+k)}{b}\right|\le \sum_{n=2}^\infty \left(\frac{\gamma}{b}\right)^{n-1}|\cos (2\pi x_n)|\le \frac{\gamma}{b-\gamma},$$
where $x_n=(x+k_1+k_2b+\cdots+k_n b^{n-1})/b^n$.
Similarly,
$$\left|-\frac{S(x,\textbf{l})}{2\pi}-\sin \frac{2\pi(x+l)}{b}\right|\le \frac{2\gamma \Delta_{b,\gamma}}{1-\gamma^2},$$
 and
$$\left|-\frac{bS'(x,\textbf{l})}{4\pi^2}-\cos \frac{2\pi(x+l)}{b}\right|\le \frac{\gamma}{b-\gamma}.$$
Therefore,
$$\left|F(x)-\left(\sin \frac{2\pi(x+k)}{b}-\sin \frac{2\pi(x+l)}{b}\right)\right|\le \frac{2\gamma\Delta_{b,\gamma}}{1-\gamma^2},$$
$$\left|\frac{b F'(x)}{2\pi}-\left(\cos \frac{2\pi(x+k)}{b}-\cos \frac{2\pi(x+l)}{b}\right)\right|\le \frac{2\gamma}{b-\gamma}.$$
Substituting $x=x^*$ gives us (\ref{eqn:c0bd}) and (\ref{eqn:c1bd}). The inequality (\ref{eqn:sumC0C1}) follows from these two inequalities and the following
$$(\cos x-\cos y)^2+(\sin x-\sin y)^2= 4\sin^2 \frac{y-x}{2}.$$
\end{proof}

Pick up $z\in [0,1]$ such that $e(1,z)=e(1)$ and pick up $k\in \{0,1,\ldots, b-1\}$ such that
$$\#\{l\in\{0,1,\ldots, b-1\}: (k,l)\in E(1, z)\}=e(1).$$
Let $k_1, k_2,\ldots, k_{e(1)}$ be all the elements in $E(1,z)$, arranged in such a way that
$$\sin (2\pi x_1)\le \sin (2\pi x_2)\le \cdots\le \sin (2\pi x_{e(1)}),$$
where $x_i=(z+k_i)/b$.

\begin{lemma}
Under the above circumstances, the following holds:
\begin{enumerate}
\item For each $1\le i<j\le e(1)$,  we have
\begin{equation}\label{eqn:sinlower1}
|\sin (2\pi x_i)-\sin (2\pi x_j)|\ge\frac{2\theta_1(b,\gamma)}{b},
\end{equation}
where
\begin{equation}\label{eqn:theta1}
\theta_1(b,\gamma)=\sqrt{\max\left(0, \left(b\sin \frac{\pi}{b}\right)^2-\frac{4\gamma^2 b^2}{(b-\gamma)^2}\right)}.
\end{equation}
\item If $k_i=k$ or $k_j=k$, then
\begin{equation}\label{eqn:sinlower0}
|\sin (2\pi x_i)-\sin (2\pi x_j)|\ge \frac{2\theta_0(b,\gamma)}{b},
\end{equation}
where
\begin{equation}\label{eqn:theta0}
\theta_0(b,\gamma)=\sqrt{\max\left(0, \left(b\sin \frac{\pi}{b}\right)^2-\frac{\gamma^2 b^2}{(b-\gamma)^2}\right)}.
\end{equation}
\item If $k_i-k_j\not=\pm 1 \mod b$, then
\begin{equation}\label{eqn:sinlower2}
|\sin (2\pi x_i)-\sin (2\pi x_j)|\ge \frac{2\theta_2(b,\gamma)}{b},
\end{equation}
where
\begin{equation}\label{eqn:theta2}
\theta_2(b,\gamma)=\sqrt{\max\left(0, \left(b\sin\frac{2\pi}{b}\right)^2-\frac{4\gamma^2 b^2}{(b-\gamma)^2}\right)}.
\end{equation}
\end{enumerate}
\end{lemma}
\begin{proof}
For each $1\le i<j\le e(1)$, we have
\begin{equation}\label{eqn:cos+sin}
|\cos (2\pi x_i)-\cos (2\pi x_j)|^2 +|\sin (2\pi x_i)-\sin (2\pi x_j)|^2 = 4 \sin^2 (\pi(x_i-x_j))\ge 4\sin^2 \frac{\pi}{b}.
\end{equation}
If $k_i=k$ or $k_j=k$ then by (\ref{eqn:c1bd}), the inequality (\ref{eqn:sinlower0}) follows.

In general, from (\ref{eqn:c1bd}), we obtain
\begin{equation}\label{eqn:coslower'}
\left|\cos (2\pi x_i)-\cos (2\pi x_j)\right|\le \frac{4\gamma}{b-\gamma},
\end{equation}
which, together with (\ref{eqn:cos+sin}), implies (\ref{eqn:sinlower1}).

If $k_i-k_j\not= \pm 1\mod b$, then
\begin{equation}
|\cos (2\pi x_i)-\cos (2\pi x_j)|^2 +|\sin (2\pi x_i)-\sin (2\pi x_j)|^2 \ge 4\sin^2\frac{2\pi}{b},
\end{equation}
which, together with (\ref{eqn:coslower'}), implies (\ref{eqn:sinlower2}).
\end{proof}

\subsection{The case when $b\ge 6$}\label{subsec:bge6}
We shall prove Theorem~\ref{thm:blarge} in the case $b\ge 6$. We separate the argument in two propositions.
\begin{prop}
Assume $b\ge 6$. If there exists $1\le i<e(1)$ such that $k_{i+1}-k_i\not =\pm 1\mod b$, then $e(1)<\gamma b$.
\end{prop}
\begin{proof} Under assumption, we have
\begin{multline*}
\frac{2\theta_2(b,\gamma)}{b}+(e(1)-2) \frac{2\theta_1(b,\gamma)}{b}\le \sum_{i=1}^{e(1)-1}\left(\sin (2\pi x_{i+1})-\sin (2\pi x_i)\right)\\
=\sin (2\pi x_{e(1)})-\sin (2\pi x_1)\le 2,
\end{multline*}
and so
\begin{equation}\label{eqn:e*theta12}
b\ge (e(1)-2) \theta_1(\gamma,b) +\theta_2(\gamma, b).
\end{equation}
We may assume
\begin{equation}
\label{eqn:gammatheta12}
\gamma \le \frac{b-\theta_2(b,\gamma)+2\theta_1(b,\gamma)}{\theta_1(b,\gamma)b},
\end{equation}
for otherwise we are done.

Note that $t\mapsto \sin t/t$ is monotone decreasing in $[0, \frac{\pi}{2})$. Since $b\ge 6$, we have
$$\theta_1(6,\gamma) \ge \theta_1(6,1)=1.8\text{ and }\theta_2(6,\gamma)\ge \theta_2(6,1)=\sqrt{21.24}> 4.$$
By (\ref{eqn:gammatheta12}), it follows that
$$\gamma \le \frac{b-4+2\cdot 1.8}{1.8b}< \frac{5}{9}.$$
Therefore, we have
\begin{equation}\label{eqn:thetanum}
\theta_1(6,\gamma) \ge \theta_1\left(6,5/9\right)>2.5\text{ and }\theta_2(6,\gamma)\ge \theta_2\left(6,5/9 \right)>5,
\end{equation}
and hence
\begin{equation}\label{eqn:gammanum}
\gamma \le \frac{b-5+2\cdot 2.5}{2.5 b}=\frac{2}{5}.
\end{equation}
Moreover, by (\ref{eqn:e*theta12}),
\begin{equation}\label{eqn:be1num}
b>2.5 e(1).
\end{equation}

{\bf Case 1.} $e(1)\le 3$.

Indeed, this is clear if $e(1)=1$ as we assume $\gamma b>1$. If $e(1)=2$ or $3$, then there exists $i, j$ such that $k_i=k\not=k_j$,
and
$$|\sin (2\pi x_i)-\sin (2\pi x_j)|\ge \frac{2\theta_2(b,\gamma)}{b}.$$
On the other hand, since $(k_i, k_j)\in E(1,z)$, we have
$$|\sin (2\pi x_i)-\sin (2\pi x_j)|\le \frac{2\gamma}{1-\gamma}.$$
Therefore,
$$\frac{2\gamma}{1-\gamma}\ge \frac{2\theta_2(b,\gamma)}{b},$$
which, together with (\ref{eqn:thetanum}) and (\ref{eqn:gammanum}),  implies that
$$\gamma b\ge  (1-\gamma)\theta_2(b,\gamma)> 5(1-2/5)=3\ge e(1).$$

{\bf Case 2.} $e(1)\ge 4$.

By (\ref{eqn:be1num}), we have $b>2.5 e(1)\ge 10$. Since $b$ is an integer, this implies $b\ge 11$.
Thus
$\theta_1(b,\gamma)\ge \theta_1(11, 2/5)>2.9$ and $\theta_2(b,\gamma)\ge \theta_2(11, 2/5)>5.8$, where we use the numerics: $\sin (\pi/11)>0.28$ and $\sin (2\pi/11)>0.54.$ By (\ref{eqn:e*theta12}) and (\ref{eqn:gammatheta12}), we obtain
\begin{equation}\label{eqn:b2.9}
b\ge 2.9 e(1) \text{ and } \gamma \le \frac{10}{29}.
\end{equation}

Let us first consider the case $e(1)=4$. Then by (\ref{eqn:b2.9}), we have $b\ge 12$, hence $\theta_1(b, \gamma)\ge \theta_1(12, 10/29)>3$.
On the other hand, there exists $1\le i,j\le 4$ such that $k_i=k$ and $|i-j|\ge 2$. Thus
\begin{equation}\label{eqn:bge6e1=4}
\frac{2\gamma}{1-\gamma}\ge |\sin (2\pi x_i)-\sin (2\pi x_j)|\ge \frac{4\theta_1(\gamma,b)}{b}.
\end{equation}
Therefore,
$\frac{2\gamma}{1-\gamma}> \frac{12}{b},$
which implies that
$$\gamma b> 6(1-\gamma).$$
If $\gamma\le 1/3$, then it follows that $\gamma b> 4=e(1)$. If $\gamma>1/3$, then $\gamma b>12\cdot (1/3)= 4=e(1).$

Now let us assume $e(1)\ge 5$. Then by (\ref{eqn:b2.9})
we have $b>2.9 e(1)\ge 14.5$ which implies that $b\ge 15$. Then
$$\theta_1(\gamma,b)\ge \theta_1(15, 10/29)>3\text{ and }\theta_2(\gamma, b)\ge \theta_2(15, 10/29)>6.$$ By (\ref{eqn:gammatheta12}), it follows that $\gamma<1/3$.
Since
$$\frac{4\gamma}{1-\gamma}\ge |\sin (2\pi x_{e(1)})-\sin (2\pi x_1)|\ge \frac{2\theta_1}{b} (e(1)-2) +\frac{2\theta_2}{b}> \frac{6e(1)}{b}$$
we obtain
$$\gamma b> 6(1-\gamma)e(1)/4>e(1).$$
\end{proof}

\begin{prop}
Assume $b\ge 6$. If $k_{i+1}-k_i=\pm 1\mod b$ for each $i\in \{1,2,\ldots, e(1)-1\}$ then $e(1)<\gamma b$.
\end{prop}
\begin{proof}
For definiteness of notation, let us assume $k_2-k_1=1\mod b$. Then since $k_1,k_2,\cdots, k_{e(1)}\in \{0,1,\ldots, b-1\}$, we have
$k_{i+1}-k_i=1\mod b$ for each $1\le i<e(1)$. Put $y^i=\pi(2k_1+2i-1+2 z)/b$. Then for each $1\le i<e(1)$,
$$\cos (2\pi x_{i+1})-\cos (2\pi x_i)=-2\sin \frac{\pi}{b}\sin y^i,$$
and $$0\le \sin (2\pi x_{i+1})-\sin (2\pi x_i)=2\sin \frac{\pi}{b}\cos y^i.$$
By (\ref{eqn:coslower'}), it follows that
$$\left|\sin  y^i \right|\le \frac{2\gamma}{(b-\gamma)\sin (\pi/b)}< 0.8\gamma,$$
where we use
$$(b-\gamma)\sin (\pi/b)> (b-1) \sin (\pi/b)\ge 5\sin (\pi/6)=5/2.$$
Therefore, $$y^i\in \bigcup_{n\in\Z} (n\pi-\arcsin (0.8\gamma),n\pi+\arcsin (0.8\gamma)).$$
For each $1\le i<e(1)-1$, $y^i$ and $y^{i+1}$ must lie in the same component of the last set, since
$$y^{i+1}-y^i= \frac{2\pi}{b}\le \frac{\pi}{3}< \pi -2\arcsin (0.8\gamma).$$
 Therefore, there exists $n_0\in\Z$ such that
$$y^i-n_0\pi\in (-\arcsin (0.8\gamma),\arcsin (0.8\gamma))\text{ for each } i\in\{1,2,\ldots, e(1)-1\}.$$
Consequently,
\begin{equation}\label{eqn:e*2nd}
e(1)-2=\frac{y^{e(1)-1}-y^1}{2\pi/b}< \frac{2\arcsin(0.8\gamma)}{2\pi/b}\le 0.4 \gamma b,
\end{equation}
where we used $\arcsin t\le \pi t/2$ for each $t\in [0,1]$.
If $2+0.4\gamma b\le \gamma b$, then we are done. So assume the contrary. Then $\gamma b <10/3$ and hence $e(1)-2<4/3$. Therefore $e(1)\le 3$. If $\gamma>1/2$, then $e(1)<\gamma b$ holds. So assume $\gamma\le 1/2$.
Then
$$\frac{\arcsin (0.8\gamma)}{0.8\gamma}\le \frac{\arcsin (0.5)}{0.5}=\frac{\pi}{3},$$
and hence (\ref{eqn:e*2nd}) improves to the following $e(1)-2< 4\gamma b/15$. If $2+4\gamma b/15\le \gamma b$ then we are done. So assume $2+4\gamma b/15>  \gamma b$. Then $\gamma b < 30/11$ and hence $e(1)-2<4\gamma b/15<1$. It follows that $e(1)=1$ or $2$.
If $\gamma b>2$ then $e(1)<\gamma b$. So assume $\gamma b\le 2$. To complete the proof we need to show $e(1)=1$.
By (\ref{eqn:sumC0C1}), it suffices to show
\begin{equation*}
\left(\frac{2\gamma}{b-\gamma}\right)^2+ \left(\frac{2\gamma\Delta_{b,\gamma}}{1-\gamma^2}\right)^2<
4\sin ^2\frac{\pi }{b}.
\end{equation*}
Since $\gamma b\le 2$, we are reduced to show
\begin{equation}\label{eqn:bge6finish}
\frac{16}{(b^2-2)^2} +\frac{16}{(b-2)^2}\left(\frac{\Delta_{b,2/b}}{1+2/b}\right)^2< 4\sin ^2\frac{\pi }{b}.
\end{equation}
In the case $b=6$, by  (\ref{eqn:Delta6}), $\Delta_{6,1/3}\le \max (0.99+1/3, 1+0.972/3)=1.324$,
then an easy numerical calculation shows that the left hand side of (\ref{eqn:bge6finish}) is less than the right hand side which is equal to $1$.
Assume now $b\ge 7$.  Using $\Delta_{b,2/b}\le 1+2/b$, we are further reduced to show
\begin{equation}\label{eqn:bge6finish1}
\frac{4b^2}{(b^2-2)^2} +\frac{4b^2}{(b-2)^2}< b^2\sin ^2\frac{\pi }{b}.
\end{equation}
Note that the left hand side is decreasing in $b$ and the right hand side is increasing in $b$. Thus it suffices to verify this inequality in the case $b=7$, which is an easy exercise.
\end{proof}

\subsection{The case $b=5$}
We use $\sin (\pi/5)=\sqrt{10-2\sqrt{5}}/4$.
By (\ref{eqn:sinlower1}), for each $1\le i< e(1)$, since $\gamma<1$,
$$|\sin (2\pi x_i)-\sin (2\pi x_{i+1})|\ge \sqrt{4\sin^2\frac{\pi}{5}-(4/4)^2}=(\sqrt{5}-1)/2>0.6.$$
Moreover, by (\ref{eqn:sinlower0}) if either $k_i=k$ or $k_{i+1}=k$, then
$$|\sin (2\pi x_i)-\sin (2\pi x_{i+1})|\ge \sqrt{4\sin^2\frac{\pi}{5}-(2/4)^2}=\sqrt{9-\sqrt{5}}/2>1.$$
Thus
$$2\ge |\sin (2\pi x_{e(1)})-\sin (2\pi x_1)|> 1+ 0.6 (e(1)-2),$$
which implies $e(1)\le 3$, since $e(1)$ is an integer. If $\gamma>3/5$ then $e(1)<\gamma b$. Assume now $\gamma\le 3/5$. Then by (\ref{eqn:sinlower1}), for each $1\le i<e(1)$,
$$|\sin (2\pi x_i)-\sin (2\pi x_{i+1})|\ge \sqrt{4\sin^2\frac{\pi}{5}-(3/5)^2}=\sqrt{(4.64-\sqrt{5})/2}>1.$$
Thus $2\ge 1.1+(e(1)-2)$ which implies $e(1)\le 2$.

\subsection{The case $b=4$}
We use $\sin (\pi/4)=\sqrt{2}/2$. By (\ref{eqn:sinlower1}), for each $1\le i<e(1)$, since $\gamma<1$,
$$|\sin (2\pi x_i)-\sin (2\pi x_{i+1})|\ge \sqrt{4\sin^2\frac{\pi}{4}-(4/3)^2}=\frac{\sqrt{2}}{3}.$$
Moreover, by (\ref{eqn:sinlower0}), if $k_i=k$ or $k_{i+1}=k$, then 
$$|\sin (2\pi x_i)-\sin (2\pi x_{i+1})|\ge \sqrt{4\sin^2\frac{\pi}{4}-(2/3)^2}=\frac{\sqrt{14}}{3}.$$
Thus
$$2\ge |\sin (2\pi x_{e(1)})-\sin (2\pi x_1)|\ge \frac{\sqrt{14}}{3}+(e(1)-2) \frac{\sqrt{2}}{3},$$
which implies  $e(1)\le 3$. Therefore, either $e(1)<\gamma b$ or $\gamma\le 3/4$. Assume the latter. Then by (\ref{eqn:sinlower1}),
for each $1\le i<e(1)$ we have
$$|\sin (2\pi x_i)-\sin (2\pi x_{i+1})|\ge \sqrt{4\sin^2\frac{\pi}{4}-1}=1.$$
Thus $2\ge \sqrt{14}/3 +(e(1)-2)$, which implies $e(1)\le 2$.

\subsection{The case $b=3$}
We use $\sin (\pi/3)=\sqrt{3}/2$. We claim that for each $z\in [0,1]$, $E(1,z)\not=\{0,1,2\}^2$, so that by Lemma~\ref{lem:1miss}, $\sigma(1)\le \sqrt{2}+1$.
Otherwise, there exists $z\in [0,1]$ such that $E(1,z)=\{0,1,2\}^2$.  Using the notation introducded above, for any $1\le i<j\le 3$, as in (\ref{eqn:sinlower0}), we have
$$|\sin (2\pi x_i)-\sin (2\pi x_j)|\ge \sqrt{4\sin^2\frac{\pi}{3}-\frac{4}{4}}=\sqrt{2},$$
which contradictis the fact $$2\ge \sin (2\pi x_3)-\sin (2\pi x_1)=|\sin (2\pi x_3)-\sin (2\pi x_2)|+|\sin (2\pi x_2)-\sin (2\pi x_1)|.$$

Assume $\sigma(1)\ge \gamma b$. Then $\gamma < (1+\sqrt{2})/3<0.81$. Keep the notation $x_j$, $e(1)$ as above. By (\ref{eqn:sinlower1}), for each $1\le i<e(1)$ we have
$$|\sin (2\pi x_i)-\sin (2\pi x_{i+1})|\ge \sqrt{4\sin^2\frac{\pi}{3}-\frac{16\gamma^2}{(3-\gamma)^2}}>0.9.$$
By (\ref{eqn:sinlower0}), if $k_i=k$ or $k_{i+1}=k$, then $|\sin (2\pi x_i)-\sin (2\pi x_{i+1})|\ge \sqrt{2}.$
Thus $2\ge \sqrt{2}+(e(1)-2) \cdot 0.9$ which implies that $e(1)\le 2$.

\section{Proof of Theorem~\ref{thm:reduced}: The case $b=3,4,5$}

In this section, we shall prove Theorem~\ref{thm:reduced} in the case $b\in \{3,4,5\}$. We shall need the following improvment of Lemma~\ref{lem:1stest}.
\begin{lemma}\label{lem:2ndest}
Let $x^*\in [0,1/2]$ and $0\le k<l<b$ be such that $(k,l)\in E(1, x^*)$. Then for any $\kappa\in (0,1)$, one of the following holds:
either
\begin{equation}\label{eqn:c1est'}
\left|\cos \frac{2\pi (x^*+k)}{b}-\cos \frac{2\pi (x^*+l)}{b}\right|\le \frac{2\gamma\sqrt{1-\kappa^2}}{b}+ \frac{2\gamma^2}{b(b-\gamma)},
\end{equation}
or
\begin{equation}\label{eqn:c0est'}
\left|\sin \frac{2\pi (x^*+k)}{b}-\sin \frac{2\pi (x^*+l)}{b}\right|\le 2 \kappa \gamma+ \frac{2\gamma^2}{1-\gamma}.
\end{equation}
\end{lemma}
\begin{proof} By Proposition~\ref{prop:compact} (1), there exist $\textbf{k}=\{k_n\}_{n=1}^\infty$ and $\textbf{l}=\{l_n\}_{n=1}^\infty$ in $\mathcal{A}^{\Z^+}$ with $k_1=k$ and $l_1=l$ and such that  the function
\begin{equation}\label{eqn:mapF}
F(x)=-\frac{1}{2\pi} \left( S(x,\textbf{k})-S(x,\textbf{l})\right),
\end{equation}
has a multiple zero at $x=x^*$. Let
\begin{equation}\label{eqn:xnyn}
x_n=\frac{x+k_1+ bk_2+\cdots +b^{n-1} k_n}{b^n}, \, y_n=\frac{x+l_1+bl_2+\cdots +b^{n-1} l_n}{b^n},
\end{equation}
and let
$$P_n(x)=\sin (2\pi x_n)-\sin (2\pi y_n), \,\, Q_n(x)=\cos (2\pi x_n)-\cos (2\pi y_n).$$
Since $F(x^*)=\sum_{n=1}^\infty \gamma^{n-1} P_n(x^*)$, we have
$$|P_1(x^*)|\le \gamma |P_2 (x^*)|+\sum_{n=3}^\infty 2\gamma^{n-1}=\gamma |P_2(x^*)|+\frac{2\gamma^2}{1-\gamma}.$$
If $|P_2(x^*)|\le 2\kappa$, then this implies that (\ref{eqn:c0est'}) holds. Assume $|P_2(x^*)|>2\kappa$. Since
$P_2(x^*)^2+ Q_2(x^*)^2\le 4$, we have $|Q_2(x^*)|\le 2\sqrt{1-\kappa^2}$. Since $0=\frac{ b F'(x^*)}{2\pi}= \sum_{n=1}^\infty (\gamma/b)^{n-1} Q_n(x^*)$, we conclude
$$|Q_1(x^*)|\le \frac{\gamma}{b} |Q_2(x^*)|+\sum_{n=3}^\infty 2\left(\frac{\gamma}{b}\right)^{n-1}\le  \frac{2\gamma\sqrt{1-\kappa^2}}{b}+ \frac{2\gamma^2}{b(b-\gamma)}, $$
which is (\ref{eqn:c1est'}).
\end{proof}

\subsection{The case $b=5$}
By Theorem~\ref{thm:blarge} (ii), to complete the proof of Theorem~\ref{thm:reduced} in the case $b=5$, it suffices to prove the following.
\begin{theorem}\label{thm:b=5} Assume $b=5$ and $e(1)=2$. Then $\sigma (1)<5\gamma$.
\end{theorem}

\begin{lemma} \label{lem:b=5.1}
Assume $\gamma \le 2/5$. Let $0\le x^*\le 1/2$ and $0\le k<l<5$ be such that $(k,l)\in E(1,x^*)$. Then
either
$$0\le x^*< 1/10 \text{ and }(k, l)=(2,3),$$
or $$2/5< x^*\le 1/2\text{ and } (k, l)=(0,4).$$
\end{lemma}

\begin{proof} Put $x^*(k)=(x^*+k)/5$, $x^*(l)=(x^*+l)/5$, and $y^*= \pi (2x^*+k+l)/5$. We shall use Lemma~\ref{lem:1stest} and Lemma~\ref{lem:2ndest} to prove
\begin{equation}
|\sin y^*|< \sin \frac{\pi}{25}
\end{equation}
which implies the statement.

By (\ref{eqn:c1bd}) and (\ref{eqn:sumC0C1}) in Lemma~\ref{lem:1stest}, we have
\begin{equation}\label{eqn:b=5.1cos}
|\cos (2 \pi x^*(k))-\cos (2\pi x^*(l))|\le \frac{2\gamma}{5-\gamma}=\frac{4}{23},
\end{equation}
and
$$4\sin^2\frac{\pi(l-k)}{5}\le \left(\frac{2\gamma}{1-\gamma}\right)^2+\left(\frac{2\gamma}{5-\gamma}\right)^2<\left(\frac{4}{3}\right)^2+\left(\frac{4}{23}\right)^2<2.$$
The latter inequality implies that $l-k=\pm 1\mod 5$.

Let $\kappa=\sqrt{2}/2$. Let us show that the inequality (\ref{eqn:c0est'}) does not hold. Indeed, otherwise, we would have
$$|\sin (2\pi x^*(k))-\sin (2\pi x^*(l))|\le \gamma \sqrt{2}+\frac{2\gamma^2}{1-\gamma}<1.1,$$
which together with (\ref{eqn:b=5.1cos}) would imply that
\begin{multline*}
1.38\cdots=4\sin^2\frac{\pi}{5}=4 \sin^2\frac{\pi(l-k)}{5}
\\=|\cos (2 \pi x^*(k))-\cos (2\pi x^*(l))|^2 + |\sin (2\pi x^*(k))-\sin (2\pi x^*(l))|^2\\
< (4/23)^2+ 1.1^2< 1.3,
\end{multline*}
which is absurd.

Therefore the inequality (\ref{eqn:c1est'}) holds. It follows that
$$2\sin \frac{\pi}{5}|\sin y^*|=|\cos (2 \pi x^*(k))-\cos (2\pi x^*(l))|\le \frac{\gamma\sqrt{2}}{5}+ \frac{2\gamma^2}{5(5-\gamma)}<0.128,$$
and hence $|\sin (y^*)|<0.11<\sin (\pi/25)$.
\end{proof}

\begin{lemma}\label{lem:b=5.2} If $\gamma \le (\sqrt{5}+1)/10$, then $e(1)=1$.
\end{lemma}
\begin{proof} Suppose $(k,l)\in E(1,x)$. Then by (\ref{eqn:sumC0C1}) in Lemma~\ref{lem:1stest}, we obtain
$$4\sin^2\frac{(l-k)\pi}{5}<  \left(\frac{2\gamma}{5-\gamma}\right)^2+ \left(\frac{2\gamma}{1-\gamma}\right)^2< 4\sin^2\frac{\pi}{5}$$
which implies that $k=l$.
\end{proof}

\begin{proof}[Proof of Theorem~\ref{thm:b=5}]
If $\gamma>2/5$, then $\sigma (1)\le e(1)=2<5\gamma$. Assume now $\gamma\le 2/5$ so that Lemma~\ref{lem:b=5.1} applies. By Proposition~\ref{prop:compact} (3), there exist $\eps>0$ and $\delta>0$ such that if $(k,l)\in E(1,x;\eps,\delta)$ for some $x\in [0,1/2]$, then we have
either $x\in [0,1/10), (k,l)=(2,3)$ or $x\in (2/5, 1/2], (k,l)=(0,4)$.

Let $K=[1/10,2/5]\cup [3/5, 9/10]$. Then by Lemma~\ref{lem:symmetric}, $e(1,x;\eps,\delta)=1$ for all $x\in K$ and $e(1,x;\eps,\delta)\le 2$ for all $x\in [0,1)$, so the condition (i) in Lemma~\ref{lem:goldenratio} is satisfied (for $q=1$). Let us prove that the condition (ii) is satisfied. Let $x\in [0,1)\setminus K$ and $0\le k<l<5$ be such that $(k,l)\in E(1,x;\eps,\delta)$. We need to check either $x(k)\in K$ or $x(l)\in K$. Indeed, by symmetry (Lemma~\ref{lem:symmetric}), it suffices to consider the case $x\in [0,1/2]\setminus K$; while  for $x\in [0,1/10)$, we have $(k,l)=(2,3)$ and $x(3)\in K$ and for $x\in (2/5, 1/2]$ we have $(k,l)=(0,4)$ and $x(4)\in K$.  Having verified the conditions in Lemma~\ref{lem:goldenratio}, we conclude
$\sigma(1)\le (\sqrt{5}+1)/2$. By Lemma~\ref{lem:b=5.2}, $\gamma> (\sqrt{5}+1)/10$ since we assume $e(1)=2$. Thus $\sigma(1)< 5\gamma$.
\end{proof}

\subsection{The case $b=4$}
By Theorem~\ref{thm:blarge} (ii), to complete the proof of Theorem~\ref{thm:reduced} in the case $b=4$, it suffices to prove the following.
\begin{theorem}\label{thm:b=4} Assume $b=4$ and $e(1)=2$. Then $\sigma(1)< b\gamma$.
\end{theorem}
First we apply Lemma~\ref{lem:1stest} and Lemma~\ref{lem:2ndest} to obtain the following estimate.
\begin{lemma}\label{lem:b=4.1.0}
Assume $\gamma \le 1/2$. Let $0\le x^*\le 1/2$ and let $0\le k<l<4$ be such that $(k,l)\in E(1,x^*)$. Then
either
$$x^*\in [3/8,1/2] \text{ and } (k,l)\in \{(0,3),(1,2)\},$$
or $$x^*\in [0,1/8] \text{ and } (k,l)=(1,3).$$
\end{lemma}
\begin{proof} By Lemma~\ref{lem:1stest},
\begin{equation}\label{eqn:b=4.1cos}
\left|\cos \frac{2\pi (x^*+k)}{4}-\cos \frac{2\pi (x^*+l)}{4}\right| \le \frac{2\gamma}{4-\gamma}\le \frac{2}{7}.
\end{equation}
Let us apply Lemma~\ref{lem:2ndest} with $\kappa=1/3$. We claim that (\ref{eqn:c0est'}) does not hold. Indeed, otherwise,
$$\left|\sin  \frac{2\pi (x^*+k)}{4}-\sin \frac{2\pi (x^*+l)}{4}\right|\le \frac{4}{3},$$
which together with (\ref{eqn:b=4.1cos}) would imply that
\begin{multline*}
2\le 4\sin^2\frac{\pi (l-k)}{4}\\ = \left(\cos \frac{2\pi (x^*+k)}{4}-\cos \frac{2\pi (x^*+l)}{4}\right)^2 +\left(\sin  \frac{2\pi (x^*+k)}{4}-\sin \frac{2\pi (x^*+l)}{4}\right)^2\\
\le \left(\frac{2}{7}\right)^2+\left(\frac{4}{3}\right)^2<2,
\end{multline*}
which is absurd. Therefore, the inequality (\ref{eqn:c1est'}) holds with $\kappa=1/3$, which implies that
$$2\left|\sin \frac{\pi(l-k)}{4}\right| \left|\sin \left(\frac{2\pi x^*}{4}+ \frac{\pi (k+l)}{4}\right)\right|\le \frac{\sqrt{2}}{6}+\frac{1}{28}.$$
Consequently,
$$\left|\sin \left(\frac{2\pi x^*}{4}+ \frac{\pi (k+l)}{4}\right)\right|\le \frac{1}{6}+\frac{1}{28\sqrt{2}}<\sin\frac{\pi}{16}.$$
Since $2\pi x^*/4\in [0, \pi/4]$ the lemma follows.
\end{proof}

A bit more careful analysis shows that the seocnd alternative in the lemma above never happens.

\begin{lemma} \label{lem:b=4.1.1}
Assume $\gamma \le 1/2$. Then for any $x^*\in [0,1/2]$, $(1,3)\not\in E(1,x^*)$.
\end{lemma}

\begin{proof} We shall prove this lemma by contradiction. Assume $(1,3)\in E(1,x^*)$. Then
there exists $\textbf{k}=\{k_n\}_{n=1}^\infty$ and $\textbf{l}=\{l_n\}_{n=1}^\infty$ with $k_1=1$ and $l_1=3$ and such that  the function
\begin{equation*}
F(x)=-\frac{1}{2\pi} \left( S(x,\textbf{k})-S(x,\textbf{l})\right)
\end{equation*}
has a multiple zero at $x^*$. Let
\begin{equation*}
x_n=\frac{x+k_1+ 4k_2+\cdots +4^{n-1} k_n}{4^n}, \, y_n=\frac{x+l_1+4l_2+\cdots +4^{n-1} l_n}{4^n},
\end{equation*}
and let
$$P_n(x)=\sin (2\pi x_n)-\sin (2\pi y_n), \,\, Q_n(x)=\cos (2\pi x_n)-\cos (2\pi y_n).$$
Then $F(x)=\sum_{n=1}^\infty \gamma^{n-1}  P_n(x)$. Since $F(x^*)=0$, this gives us
$$|P_1(x^*)+\gamma P_2(x^*)|\le \sum_{n=3}^\infty\gamma^{n-1} |P_n(x^*)|\le 1.$$
Note that
$$P_2(x)\ge -\cos \frac{\pi (1+x)}{4}-\cos \frac{\pi (1-x)}{4}=-2\cos \frac{\pi}{4}\cos \frac{\pi x}{4}\ge -\sqrt{2}.$$
Therefore
$$P_1(x^*)\le 1-\gamma P_2(x^*)\le 1+\frac{\sqrt{2}}{2}.$$
As in the previous lemma, $|Q_1(x^*)|\le 2/7$. Since $P_1(x^*)>0$, we have
$$P_1(x^*)^2 +Q_1(x^*)^2 \le (2/7)^2+ (1+\sqrt{2}/2)^2<4.$$
However, the left hand of the inequality is equal to $4$, a contradiction!
\end{proof}

\begin{lemma}\label{lem:b=4.2} If $\gamma\le (\sqrt{5}+1)/8$ then $e(1)=1$.
\end{lemma}
\begin{proof}
For $x\in [0,1]$ and $0\le k<l<4$, if $(k,l)\in E(1,x)$, then by the inequality (\ref{eqn:sumC0C1}) in Lemma~\ref{lem:1stest}, we have
$$4\sin^2\frac{\pi(l-k)}{4}<  \left(\frac{2\gamma}{4-\gamma}\right)^2+ \left(\frac{2\gamma}{1-\gamma}\right)^2<2,$$
which implies that $l=k$.
\end{proof}

\begin{proof}[Proof of Theorem~\ref{thm:b=4}] If $\gamma>1/2$, then $\sigma (1)\le e(1)=2<4\gamma$. So assume $\gamma \le 1/2$.
By Lemmas~\ref{lem:b=4.1.0} and ~\ref{lem:b=4.1.1} and Proposition~\ref{prop:compact} (3), there exist $\eps>0$ and $\delta>0$ such that if $x\in [0,1/2]$ and $0\le k<l<4$ are such that $(k,l)\in E(1,x;\eps,\delta)$ then $3/8<x\le 1/2$ and $(k,l)\in \{(0,3), (1,2)\}$. Note that for $x\in (3/8, 1/2]$, we have $x(0), x(1)\in [0, 3/8]$.
By Lemma~\ref{lem:symmetric}, it is then easy to check that the conditions of Lemma~\ref{lem:goldenratio} are satisfied for $K=[0,3/8]\cup [5/8,1)$ and $q=1$.
Thus $\sigma (1)\le (\sqrt{5}+1)/2$. On the other hand, since we assume $e(1)=2$, by Lemma~\ref{lem:b=4.2}, we have $\gamma >(\sqrt{5}+1)/8$. This proves that $\sigma (1)< 4\gamma$.
\end{proof}

\subsection{The case $b=3$}
In this subsection, we shall prove the following theorem, which together with Theorem~\ref{thm:blarge}, implies Theorem~\ref{thm:reduced} in the case $b=3$.
\begin{theorem}\label{thm:b=3} Assume $b=3$ and $e(1)=2$. Then $\sigma(1)<b\gamma$.
\end{theorem}

\begin{lemma} \label{lem:b=3.0} Assume $\gamma \le 2/3$. Let $0\le x^*\le 1/2$ and let $0\le k<l<3$ be such that $(k,l)\in E(1,x)$. Then
either
$$x\in [0,1/6) \text{ and } (i,j)=(1,2); $$
or $$x\in (1/3, 1/2] \text{ and } (i,j)=(0,2).$$
\end{lemma}
\begin{proof} If $(k,l)\in E(1,x^*)$, then by (\ref{eqn:c1bd}),
$$\left|\cos \frac{2\pi (x^*+k)}{3}-\cos \frac{2\pi (x^*+l)}{3}\right|\le \frac{2\gamma}{3-\gamma}\le \frac{4}{7},$$
which implies that
$$\left|\sin \left(\frac{2\pi x^*}{3} +\frac{\pi (k+l)}{3}\right)\right|\le \frac{4}{7\cdot 2|\sin \frac{\pi (l-k)}{3}|}= \frac{4}{7\sqrt{3}}<\sin \frac{\pi}{9}.$$
The statement follows.
\end{proof}

\begin{lemma}\label{lem:b=3.1} Assume $\gamma \le (\sqrt{5}+1)/6$. Let $0\le x\le 1/2$ and let $0\le k<l<3$ be such that $(k,l)\in E(1,x)$. Then
either
$$x\in [0,1/8] \text{ and } (k,l)=(1,2); $$
or $$x\in [3/8, 1/2] \text{ and } (k,l)=(0,2).$$
\end{lemma}
\begin{proof}
Under current assumption, again by (\ref{eqn:c1bd}), we have
$$\left|\cos \frac{2\pi (x^*+k)}{3}-\cos \frac{2\pi (x^*+l)}{3}\right|\le \frac{2\gamma}{3-\gamma}\le 0.44,$$
hence $$\left|\sin \left(\frac{2\pi x^*}{3} +\frac{\pi (k+l)}{3}\right)\right|\le \frac{0.44}{2|\sin \frac{\pi (l-k)}{3}|}= \frac{0.44}{\sqrt{3}}<\sin \frac{\pi}{12}.$$
Thus the statement holds.
\end{proof}

\begin{lemma}\label{lem:b=3.2}
If $3\gamma \le \sqrt{2}$, then $e(1)=1$.
\end{lemma}
\begin{proof}
By (\ref{eqn:Delta3}) and (\ref{eqn:sumC0C1}), if $(k,l)\in E(1, x^*)$ for some $0\le k,l<3$ and $x^*\in\R$, then
$$4\sin^2\frac{\pi(l-k)}{3} \le \left(\frac{2\gamma}{3-\gamma}\right)^2+\left(\frac{2\gamma (1+0.71\gamma)}{1-\gamma^2}\right)^2.$$
Since $\gamma\le \sqrt{2}/3$, the right hand side is less than $3$, which implies that $k=l$. This proves that $e(1)=1$.
\end{proof}

\begin{proof} [Proof of Theorem~\ref{thm:b=3}]
If $\gamma>2/3$, then $\sigma (1)\le e(1)=2<3\gamma$. So assume $\gamma\le 2/3$.
By Lemma~\ref{lem:b=3.0} and Proposition~\ref{prop:compact} (3), there exist $\eps>0$ and $\delta>0$ such that if $x\in [0,1/2]$, $0\le k<l<3$ are such that $(k,l)\in E(1,x;\eps,\delta)$, then either $x\in [0,1/6), (k,l)=(1,2) $ or $x\in (1/3, 1/2], (k,l)=(0,2)$. Note that for $x\in [0,1/6)$, we have $x(2)\in [2/3, 5/6]$ and for $x\in (1/3, 1/2]$, we have $x(0)\in [1/6, 1/3]$. By Lemma~\ref{lem:symmetric}, it follows that the conditions in Lemma~\ref{lem:goldenratio} are satisfied with $K=[1/6, 1/3]\cup [2/3, 5/6]$ and $q=1$. Thus $\sigma(1)\le (\sqrt{5}+1)/2$.

If $\gamma>(\sqrt{5}+1)/6$, then $\sigma (1)< 3\gamma$. Assume $\gamma\le (\sqrt{5}+1)/6$. Then by Lemma~\ref{lem:b=3.1} and Proposition~\ref{prop:compact} (3), there exist $\eps>0$ and $\delta>0$ such that if $x\in [0,1/2]$, $0\le k<l<3$ are such that $(k,l)\in E(1,x;\eps,\delta)$, then either $x\in [0,1/8), (k,l)=(1,2) $ or $x\in (3/8, 1/2], (k,l)=(0,2)$. Note that for $x\in [0,1/8)$, we have $x(1)\in [1/8, 3/8]$, $x(2)\in [5/8, 7/8]$ and for $x\in [3/8, 1/2]$ we have $x(0)\in [1/8, 3/8]$ and $x(2)\in [5/8, 7/8]$. By Lemma~\ref{lem:symmetric}, it follows that the conditions in Lemma~\ref{lem:sqrt2} are satisfied with $K=[0, 3/8]\cup [5/8,1)$ and $q=1$.
Thus $\sigma (1)\le \sqrt{2}$. Since we assume $e(1)=2$, by Lemma~\ref{lem:b=3.2}, $\gamma>\sqrt{2}/3$. This proves $\sigma (1)< 3\gamma.$
\end{proof}
\section{The case $b=2$}
This section is devoted to the proof of Theorem~\ref{thm:reduced} in the case $b=2$.  The proof is structurally similar to the cases $b=3,4,5$ which we discussed above, but it is more involved and consists of several steps.

We shall use the following notation. For any $\textbf{k},\textbf{l}\in \mathcal{A}^q$ and $x_*\in \R$, we write
$\textbf{k}\sim_{x_*} \textbf{l}$ if $(\textbf{k},\textbf{l})\in E(q,x_*)$. In order to show $\textbf{k}\not\sim_{x_*} \textbf{l}$, by Proposition~\ref{prop:compact} (1), it suffices to show
that any $\textbf{u},\textbf{v}\in\mathcal{A}^{\Z^+}$, the function $S(x, \textbf{ku})-S(x,\textbf{lv})$ does not have a multiple zero at $x_*$.

\subsection{Step 1. When $\gamma > \sqrt{2}/2$}
In this step, we shall prove
\begin{prop} \label{prop:b=2step1}
\begin{enumerate}
\item For any $\gamma\in (1/2, 1)$, $\sigma(1)\le (\sqrt{5}+1)/2$.
\item For any $\gamma \in (1/2, (\sqrt{5}+1)/4]$, $\sigma(1)\le \sqrt{2}$.
\end{enumerate}
\end{prop}

An immediate corollary is the following:
\begin{coro}\label{cor:b=2.1}
 Assume $b=2$ and $\gamma > \sqrt{2}/2$. Then $\sigma(1)<2\gamma$.
\end{coro}

The proof of this proposition relies on the following estimates.
\begin{lemma} \label{lem:b=2.1}
\begin{enumerate}
\item[(i)] For any $x\in [0,1/3]$, we have
$$(00)\not\sim_x (10), (00)\not\sim_x (11), \text{ and } (01)\not\sim_x (10).$$
\item[(ii)] If either $x\in [0,1/4]$,  or $x\in [0,1/3]$ and $\gamma \le (\sqrt{5}+1)/4$, then $(01)\not\sim_x (11)$.
\end{enumerate}
\end{lemma}
\begin{proof}
Let $\textbf{u}=\{u_n\}_{n=1}^\infty$ and $\textbf{v}=\{v_n\}_{n=1}^\infty$ be elements in $\mathcal{A}^{\Z^+}$ and let
$$F(x)=-\frac{1}{2\pi} \left(S(x, \textbf{u})-S(x,\textbf{v})\right).$$
As before, given $x\in \R$, we write
$$x_n=\frac{x}{2^n}+\frac{u_1}{2^n}+\cdots+\frac{u_n}{2}, \,\, y_n=\frac{x}{2^n}+\frac{v_1}{2^n}+\cdots+\frac{v_n}{2}$$
and write $Q_n=\cos (2\pi x_n)-\cos (2\pi y_n)$.
Then $|Q_n|\le 2$ for all $n$ and
$$G(x):=\frac{F'(x)}{\pi}= \sum_{n=1}^\infty \left(\frac{\gamma}{2}\right)^{n-1} Q_n.$$

In the following, we assume $u_1=0$ and $v_1=1$, so that $Q_1=2\cos (\pi x)$.

(i). Assume first $u_2=0$. Then  
$$Q_2= \cos \frac{\pi x}{2}+(-1)^{v_2}\sin \frac{\pi x}{2}\ge \cos \frac{\pi x}{2}-\sin\frac{\pi x}{2}>0$$
for any $x\in [0,1/2]$.
Thus
\begin{align*}
G(x)  \ge Q_1-\sum_{n=3}^\infty (\gamma/2)^{n-1} |Q_n|> 2\cos (\pi x) -1 \ge 0.
\end{align*}
whenever $0\le x\le 1/3$.
This proves that for any $x\in [0,1/3]$, $(00)\not\sim_x (10)$ and $(00)\not\sim_x (11)$.

To prove $(01)\not\sim_x (10)$ for $x\in [0,1/3]$, let $u_2=1$ and $v_2=0$.
Then
$$Q_2=\sin (\pi x/2)-\cos (\pi x/2),$$
and
\begin{multline*}
Q_3= \pm \sin (\pi x/4)\pm \cos (\pi(1+x)/4)
\ge -\sin(\pi/12)-\cos (\pi/4)>-1.
\end{multline*}
Thus
\begin{align*}
G(x)
\ge & Q_1 + \frac{\gamma}{2} Q_2+\frac{\gamma^2}{4} Q_3+\sum_{n=3}^\infty\frac{\gamma^n}{2^n} Q_{n+1}\\
\ge & g(x)-\frac{\gamma^2}{4} -\frac{\gamma^3}{4-2\gamma}
\ge g(x)-\frac{3}{4},
\end{align*}
where $$g(x)= 2\cos (\pi x) + \frac{1}{2}\left(\sin\frac{\pi x}{2}-\cos \frac{\pi x}{2}\right).$$
Since $$g''(x) = -2\pi^2 \cos (\pi x) -\pi^2/8 (\sin (\pi x/2) -\cos (\pi x/2)<0$$ holds for all $x\in [0,1/3]$, we have
%
$$\min_{x\in [0,1/3]} g(x)=\min (g(0), g(1/3))=\min \left(\frac{3}{2}, \frac{5-\sqrt{3}}{4}\right)>\frac{3}{4}.$$
Therefore $G>0$. 

(ii) Now let us assume $u_2=v_2=1$.
Then $$Q_2=-\sin \frac{\pi x}{2}-\cos \frac{\pi x}{2},$$
and $$Q_3= \pm \sin \frac{\pi x}{4} \pm \sin \frac{\pi (1+x)}{4}\ge -2\sin \frac{\pi (1+2x)}{8} \cos \frac{\pi}{8}.$$
Thus
$$G(x)\ge Q_1+\frac{\gamma}{2} Q_2+\frac{\gamma^2}{4} Q_3-2\sum_{n=4}^\infty \left(\frac{\gamma}{2}\right)^{n-1}\ge h_\gamma(x)-\frac{\gamma^3}{4-2\gamma},$$
where
$$h_\gamma(x)=2\cos (\pi x)-\frac{\gamma}{2} \left(\cos \frac{\pi x}{2} +\sin \frac{\pi x}{2}\right)-\frac{\gamma^2}{2} \cos \frac{\pi}{8}\sin \frac{\pi (1+2x)}{8}.$$

Assume first $x\in [0,1/4]$. Since $h_\gamma(x)$ is decreasing in both $x$ and $\gamma$, we have
$$h_\gamma (x)\ge h_1(1/4)= \sqrt{2}-\frac{\sqrt{2}}{2} \sin \frac{3\pi}{8} -\frac{1}{2} \cos\frac{\pi}{8}\sin \frac{3\pi}{16}>\frac{1}{2}>\frac{\gamma^3}{4-2\gamma},$$
hence $G(x)>0$.
This proves that $(01)\not\sim_x (11)$ for $x\in [0,1/4]$.

Assume now $\gamma\le (\sqrt{5}+1)/4=:\gamma_0$. Then by numerical calculation, we have
$$h_{\gamma_0} (1/3)=1-\frac{\gamma_0}{2}\frac{\sqrt{3}+1}{2} - \frac{\gamma_0^2}{2} \sin \frac{\pi}{8} \sin \frac{5\pi}{24}> \frac{\gamma_0^3}{4-2\gamma_0}.$$
Thus for each $0\le x\le 1/3$, we have
$$G(x)\ge h_{\gamma}(x)-\gamma^3/(4-2\gamma)\ge h_{\gamma_0}(1/3)-\gamma_0^3/(4-2\gamma_0)>0.$$
This proves that $(01)\not\sim_x (11)$ for all $x\in [0,1/3]$.
\end{proof}

\begin{proof}[Proof of Proposition~\ref{prop:b=2step1}]
By Lemma~\ref{lem:b=2.1} and Proposition~\ref{prop:compact} (3), there exist $\eps>0$ and $\delta>0$ such that $e(1,x;\eps,\delta)=1$ for $x\in [0,1/4]$.
By Lemma~\ref{lem:symmetric}, this also holds for $x\in [3/4, 1]$.
For $x\in [1/4, 1/2]$, $x(0)\in [0,1/4]$ while for $x\in [1/2, 3/4]$ we have $x(1)\in [3/4, 1]$. By Lemma~\ref{lem:goldenratio}, we obtain
$\sigma(1)\le (\sqrt{5}+1)/2$.

Assume that $\gamma \le (\sqrt{5}+1)/4$. Then Lemma~\ref{lem:b=2.1} and Proposition~\ref{prop:compact} (3), there exist $\eps>0$ and $\delta>0$ such that $e(1,x;\eps,\delta)=1$ for all $x\in [0,1/3]$ and by Lemma~\ref{lem:symmetric}, this also holds for $x\in [2/3,1]$. For $x\in (1/3, 2/3)$, both $x(0)$ and $x(1)$ belong to $[0, 1/3]\cup [2/3, 1]$. By Lemma~\ref{lem:sqrt2}, it follows that $\sigma(1)\le \sqrt{2}$.
\end{proof}

\subsection{Step 2: When $\gamma > 0.64$}
In this section we shall prove the following
\begin{prop} \label{prop:b=2step2}
Assume $\gamma \le \sqrt{2}/2$. Then $\sigma(2)\le 1.61.$
\end{prop}

As an immediate corollary of this proposition and Corollary~\ref{cor:b=2.1}, we have
\begin{coro}\label{cor:b=2.2} If $\gamma > 0.64$ then either $\sigma(1)< 2\gamma$ or $\sigma (2)< (2\gamma)^2$.
\end{coro}

The proof of Proposition~\ref{prop:b=2step2} relies on the following estimates.
\begin{lemma} \label{lem:b=2.2}
Assume $\gamma\le \sqrt{2}/2$. Then
\begin{enumerate}
\item [(i)] For any $x\in [0,1/2]$, $(00)\not\sim_x (10)$;
\item [(ii)] For any $x\in [0,2/5]$, $(01)\not\sim_x (10)$ and $(00)\not\sim_x (11)$;
\item [(iii)] For any $x\in [0,1/2]$, $(00)\not\sim_x (11)$;
\item [(iv)] For any $x\in [0,2/5]$, $(01)\not\sim_x (11)$;
\item [(v)] For any $x\in [0,2/5]$, $0\not\sim_x 1$;
\item [(vi)] For $x\in [1/5, 1/2]$, $(10)\not\sim_x (11)$.
\end{enumerate}
\end{lemma}

\begin{proof} Let $\textbf{u},\textbf{v}\in\mathcal{A}^{\Z^+}$ and let $F(x)= -(2\pi)^{-1}(S(x, \textbf{u})-S(x,\textbf{v}))$, $G(x)=F'(x)/\pi$. For given $x$, we shall use the notations $x_n$, $y_n$, $Q_n$ as in Lemma~\ref{lem:b=2.1} and let $P_n=\sin (2\pi x_n)-\sin (2\pi y_n)$.

(i) Assume $(u_1,u_2)=(0,0)$ and $(v_1,v_2)=(1,0)$. Then $Q_1=2\cos (\pi x)$ and
$Q_2=\cos \frac{\pi x}{2}+\sin \frac{\pi x}{2}$. Thus
$$G(x)=Q_1+\frac{\gamma}{2} Q_2 +\cdots \ge f(x)-\frac{\gamma^2}{2-\gamma},$$
where $$f(x)=2 \cos (\pi x)+\frac{\gamma}{2}\left(\sin\frac{\pi x}{2}+\cos \frac{\pi x}{2}\right).$$
On the interval $x\in [0,1/2]$, we have
$$f''(x)= -2\pi^2 \sin (\pi x) -\frac{\pi^2\gamma}{8} \left(\sin\frac{\pi x}{2}+\cos \frac{\pi x}{2}\right)<0.$$
Thus $$\min_{x\in [0,1/2]} f(x)=\min (f(0), f(1/2))=\min \left(2+\frac{\gamma}{2},\frac{\gamma}{\sqrt{2}}\right)=\frac{\gamma}{\sqrt{2}}
> \frac{\gamma^2}{2-\gamma},$$
hence $G(x)>0$.
Therefore, for any $x\in [0,1/2]$, $(00)\not\sim_x (11)$.

(ii) Assume either $(u_1,u_2)=(0,1)$ and $(v_1,v_2)=(1,0)$; or $(u_1,u_2)=(0,0)$ and $(v_1,v_2)=(1,1)$. 
Then $$Q_2=\pm\left(\cos \frac{\pi x}{2}-\sin \frac{\pi x}{2}\right)\ge \sin \frac{\pi x}{2}-\cos \frac{\pi x}{2}.$$
Thus
\begin{align*}
G(x) & \ge 2\cos (\pi x) - \frac{\gamma}{2} \left(\cos\frac{\pi x}{2}-\sin \frac{\pi x}{2}\right)-\frac{\gamma^2}{2-\gamma}\\
& \ge 2 \cos (\pi x)-\frac{1}{2}\sin \frac{\pi (1-2x)}{4} -\frac{4+\sqrt{2}}{14}:=g(x).
\end{align*}
For $x\in [0,2/5]$, we have
$$g''(x)=-2\pi^2 \cos (\pi x)+\frac{\pi^2}{8} \sin \frac{\pi (1-2x)}{4}<0.$$
Thus $g(x)\ge \min (g(0), g(2/5))$. Clearly, $g(0)>0$ and $g(2/5)>0$. It follows that $G(x)>0$ for all $x\in [0,2/5]$.
This proves that $(01)\not\sim_x (10)$ and $(00)\not\sim_x (11)$ for $x\in [0,2/5]$.

(iii) Assume $(u_1, u_2)=(0,0)$ and $(v_1, v_2)=(1,1)$. By (ii), it suffices to show that
$F(x)>0$ for any $x\in [2/5,1/2]$.
Note that $P_1=2\sin (\pi x)$,
$$P_2=\sin \frac{\pi x}{2}+\cos \frac{\pi x}{2}=\sqrt{2}\sin \left(\frac{\pi x}{2}+\frac{\pi}{4}\right),$$
and
$$P_3\ge -\sin \frac{\pi x}{4}-\sin \frac{\pi (1-x)}{4}=-2\sin \frac{\pi}{8}\cos \frac{\pi (1-2x)}{8}\ge -2\sin \frac{\pi}{8},$$
so
\begin{align*}
\frac{F(x)}{\gamma}=\sum_{n=1}^\infty \gamma^{n-2} P_n\ge \frac{2\sin (\pi x)}{\gamma} + \sqrt{2} \sin \left(\frac{\pi x}{2}+\frac{\pi}{4}\right)- 2\gamma \sin \frac{\pi}{8}-\frac{2\gamma^2}{1-\gamma}
\end{align*}
The right hand side is increasing in $x\in [2/5, 1/2]$ and decreasing in $\gamma$. Thus for $x\in [2/5, 1/2]$,
$$\frac{F(x)}{\gamma}\ge 2\sqrt{2} \sin \frac{2\pi}{5}+ \sqrt{2} \sin \frac{9\pi}{20}- \sqrt{2} \sin \frac{\pi}{8}-2-\sqrt{2}>0.$$

(iv) Assume $(u_1, u_2)=(0,1)$, $(v_1, v_2)=(1,1)$. By Lemma~\ref{lem:b=2.1}, we only need to show that $(01)\not\sim_x (11)$ for $x\in [1/3, 2/5]$.
Then
$$P_2=\cos \frac{\pi x}{2}-\sin \frac{\pi x}{2} \text{ and } Q_2=-\sin \frac{\pi x}{2}-\cos \frac{\pi x}{2}.$$
Put
\begin{equation}\label{eqn:0111p}
p(x)= \frac{P_1}{\gamma} + P_2=\frac{2\sin (\pi x)}{\gamma} + \left(\cos \frac{\pi x}{2}-\sin \frac{\pi x}{2}\right),
\end{equation}
\begin{equation}\label{eqn:0111q}
q(x)=Q_1+\frac{\gamma}{2} Q_2=2\cos (\pi x) - \frac{\gamma}{2}  \left(\sin\frac{\pi x}{2}+\cos \frac{\pi x}{2}\right).
\end{equation}
Since $q(x)$ is decreasing in $[0,1/2]$, we have
\begin{equation}\label{eqn:0111gvalue}
q(x)\ge q(2/5)>\frac{\sqrt{5}-1}{2} -\frac{\gamma}{\sqrt{2}}\ge \frac{\sqrt{5}-2}{2}.
\end{equation}
for each $0\le x\le 2/5$.
Since $p'=\pi q/\gamma>0$, for each $x\in [1/3, 2/5]$, we have
\begin{equation}\label{eqn:0111pvalue}
p(x)\ge p(1/3) \ge \frac{\sqrt{3}}{\gamma}+\frac{\sqrt{3}-1}{2}\ge \sqrt{6}+\frac{\sqrt{3}-1}{2}>2.81.
\end{equation}

{\bf Case 1.} Assume also $v_3=0$.
Then
$$Q_3=\pm\sin\frac{\pi x}{4}+\cos \frac{\pi (1-x)}{4}\ge \cos \frac{\pi (1-x)}{4}-\sin \frac{\pi x}{4}.$$
Since the right hand side is decreasing in $[0,1/2]$, we obtain
$$Q_3\ge \cos \frac{\pi}{8}-\sin \frac{\pi}{8}>\frac{1}{2},$$ and
$$G(x)\ge Q_1+\frac{\gamma}{2} Q_2+\frac{\gamma^2}{4} Q_3-\frac{\gamma^3}{4-2\gamma}> q(x)+\frac{\gamma^2}{8}-\frac{\gamma^3}{4-2\gamma}.$$
Thus
$$\frac{G(x)}{\gamma^2} \ge \frac{\sqrt{5}-2}{2\gamma^2}+\frac{1}{8}-\frac{\gamma }{4-2\gamma}\ge \sqrt{5}-2 +\frac{1}{8}-\frac{1}{4\sqrt{2}-2}>0.$$

{\bf Case 2.} Assume $u_3=v_3=1$.
Then $$P_3=\cos \frac{\pi (1+x)}{4}-\cos \frac{\pi x}{4}= -2\sin \frac{\pi (1+2x)}{8}\sin \frac{\pi }{8}>-\sqrt{2} \sin \frac{\pi}{8}$$

and
\begin{multline*}
Q_3=\sin \frac{\pi x}{4}-\sin \frac{\pi (1+x)}{4}=-2\sin \frac{\pi}{8} \cos \frac{\pi (1+2x)}{8}>
-\sin \frac{\pi}{4}=-\frac{\sqrt{2}}{2}.
\end{multline*}

{\bf Subcase 2.1.}
$u_4=1$ and $v_4=0$. Then
$$Q_4=\cos \frac{\pi (2-x)}{8}+\cos \frac{\pi(1-x)}{8}>\sqrt{2}>\frac{2\gamma}{2-\gamma}.$$
Therefore, for $x\in [0, 2/5]$,
\begin{align*}
G(x)& =Q_1+\frac{\gamma}{2} Q_2+ \frac{\gamma^2}{4}Q_3+\frac{\gamma^3}{8} Q_4 -\frac{\gamma^4}{8-4\gamma}\\
& \ge q(x) +\frac{\gamma^2}{4} Q_3
+\frac{\gamma^3}{8}\left(Q_4-\frac{2\gamma}{2-\gamma}\right)\\
& > q(x)-\frac{\sqrt{2}}{16}>0,
\end{align*}
where the last inequality follows from (\ref{eqn:0111gvalue}).

{\bf Subcase 2.2.} $u_4=0$. Then
$$P_4\ge \sin \frac{\pi (2-x)}{8}-\sin \frac{\pi (1-x)}{8}>0.$$
Thus
\begin{align*}
\frac{F(x)}{\gamma}& \ge \frac{P_1}{\gamma}+ P_2+\gamma P_3+\gamma^2 P_4-\frac{2\gamma^3}{1-\gamma} && > p(x)+ \gamma P_3-\frac{2\gamma^3}{1-\gamma}\\
& \ge p(x)- \sqrt{2}\gamma\sin \frac{\pi}{8}-\frac{2\gamma^3}{1-\gamma}
&& \ge p(x)-\sin \frac{\pi}{8} -\sqrt{2}-1>0.01,
\end{align*}
where the last inequality follows from (\ref{eqn:0111pvalue}).

{\bf Subcase 2.3.}
$u_4=v_4=1$. Then
$$P_4=\sin \frac{\pi (1-x)}{8}-\sin \frac{\pi (2-x)}{8}=-2\sin \frac{\pi}{16} \cos \frac{\pi (3-2x)}{16}>-2\sin \frac{\pi}{16};$$
and
$$P_5\ge -\sin \frac{\pi (2-x)}{16}-\sin \frac{\pi (1-x)}{16}=-2\sin \frac{\pi (3-2x)}{32} \cos \frac{\pi}{32}>-2\sin \frac{3\pi}{32}.$$
Therefore,
\begin{align*}
\frac{F(x)}{\gamma}& \ge p(x)+\gamma P_3+\gamma^2 P_4+\gamma^3 P_5 -\frac{2\gamma^4}{1-\gamma}\\
& \ge p(x) -\gamma\sqrt{2}\sin \frac{\pi}{8}-2\gamma^2 \sin \frac{\pi}{16}-
 2\gamma^3 \sin \frac{3\pi }{32} -\frac{2\gamma^4}{1-\gamma}\\
& \ge p(x) -\sin \frac{\pi}{8}-\sin \frac{\pi}{16}- \frac{1}{\sqrt{2}}\sin \frac{3\pi}{32}-1-\frac{\sqrt{2}}{2}
>0,
\end{align*}
where the last inequality follows from (\ref{eqn:0111pvalue}).

{\bf Case 3.} Assume $u_3=0$ and $v_3=1$. 
Then for $x\in [1/3, 2/5]$,
$$P_3=\cos \frac{\pi x}{4}+\cos \frac{\pi (1+x)}{4}=2\cos \frac{\pi}{8}\cos \frac{\pi (1+2x)}{8}>\sqrt{2} \cos \frac{\pi}{8}.$$
Thus
\begin{multline*}
\frac{1}{\gamma^2} F(x)  \ge \frac{1}{\gamma^2} P_1+ \frac{P_2}{\gamma}+P_3-\frac{2\gamma}{1-\gamma}
 > \frac{p(x)}{\gamma} + \sqrt{2} \cos \frac{\pi}{8}-\frac{2\gamma}{1-\gamma}\\
 \ge \sqrt{2}\left(p(x) +\cos \frac{\pi}{8}-2-\sqrt{2}\right)>0,
\end{multline*}
where the last inequality follows from (\ref{eqn:0111pvalue}).

(v) It follows from (i) (ii) and (iv).

(vi) It suffices to show that $0\not\sim_{x(1)} 1$ for $x\in [1/5, 1/2]$. But $x(1)\in [3/5, 1)$, so the statement follows from (v) and Lemma~\ref{lem:symmetric}.
\end{proof}

Summarizing the estimates given by Lemma~\ref{lem:b=2.1} and Lemma~\ref{lem:b=2.2}, we have
\begin{lemma}\label{lem:b=2.2*} Assume $\gamma\le \sqrt{2}/2$. Then
\begin{enumerate}
\item[(1)]  For $x\in [1/5,2/5]$, $e(2,x)=1$;
\item[(2)] For $x\in [0,1/5)$, the only possible non-trivial pairs in $E(2,x)$ are $(10,11)$ and $(11,10)$;
\item[(3)] For $x\in (2/5, 1/2]$, the only possible non-trivial pairs in $E(2,x)$ are $(01,10)$, $(01,11)$, $(10,01)$ and $(11,01)$.
\end{enumerate}
\end{lemma}
\begin{proof} By Lemma~\ref{lem:b=2.1}, $(00)\not\sim_x (01)$ for all $x\in [0,1/2]$, since $0\not\sim_{x(0)} 1$. By Lemma~\ref{lem:b=2.2} (i) and (iii), we have $(00)\not\sim_x (10)$ and $(00)\not\sim_x (11)$ for all $x\in [0,1/2]$.  For $x\in [2/5, 1/2]$, by Lemma~\ref{lem:b=2.2} (vi), we also have $(10)\not\sim_x (11)$. Thus (3) holds.
For $x\in [0,2/5]$, by Lemma~\ref{lem:b=2.2} (ii) and (iv),  $(01)\not\sim_x (10)$, $(01)\not\sim_x (11)$, so the only possible non-trivial pairs in $E(2, x)$ are $(10,11)$ and $(11,10)$. So (2) holds. If $x\in [1/5, 2/5]$, then $(10)\not\sim_x (11)$ by Lemma~\ref{lem:b=2.2} (vi).
Therefore (1) holds.
\end{proof}
\begin{proof}[Proof of Proposition~\ref{prop:b=2step2}]
We shall apply Lemma~\ref{lem:w1.7}. Let $K_0=[1/5,2/5] \cup [3/5, 4/5]$,
$K_1=[0,1/5)\cup (4/5,1)$ and $K_2=(2/5, 3/5)$. By Lemma~\ref{lem:b=2.2*}, Proposition~\ref{prop:compact} (3) and Lemma~\ref{lem:symmetric}, the conditions in Lemma~\ref{lem:w1.7} are satisfied with $q=2$ and with suitable choice of $(\eps, \delta)$. Indeed,
\begin{itemize}
\item For $x\in [1/5,2/5]$, by Lemma~\ref{lem:b=2.2*}, $e(2,x)=1$, so $e(2,x;\eps,\delta)=1$ for suitable choice of $\eps, \delta$. By Lemma~\ref{lem:symmetric}, this also holds for $x\in [3/5, 4/5]$.
\item For $x\in [0,1/5)$, take $\textbf{a}_x=(10)$ and $\textbf{b}_x=(11)$. Then $x(10), x(11)\in K_0$. By Lemma~\ref{lem:b=2.2*}, $(10, 11)$ and $(11,10)$ are the only elements in $E(2,x)$. So by Proposition~\ref{prop:compact} (3) and Lemma~\ref{lem:symmetric}, the condition (ii) in Lemma~\ref{lem:w1.7} is satisfied.
\item For $x\in (2/5, 1/2]$, let $\textbf{a}_x=(01)$, $\textbf{b}_x=(10)$ and $\textbf{c}_x=(11)$. Then $x(01), x(10)\in K_0$ and $x(11)\in K_1$; and by Lemma~\ref{lem:b=2.2*}, $(01,10)$, $(01,11)$, $(10,01)$ and $(11,01)$ are the only non-trivial pairs in $E(2, x)$. So by Proposition~\ref{prop:compact} (3) and Lemma~\ref{lem:symmetric}, the condition (iii) in Lemma~\ref{lem:w1.7} is satisfied.
\end{itemize}
Thus $\sigma (2)< 1.61.$
\end{proof}

\subsection{Step 3. When $\gamma^2 >\sqrt{2}/4 $}
\begin{prop}\label{prop:b=2step3} Assume $\gamma \le 0.64$. Then $\sigma(2)\le \sqrt{2}$.
\end{prop}
An immediate corollary of this proposition and Corollary~\ref{cor:b=2.2} is the following:
\begin{coro}\label{cor:b=2.3} If $\gamma^2 >\sqrt{2}/4$, then either $\sigma(1)<2\gamma$ or $\sigma(2)<(2\gamma)^2$.
\end{coro}
The proof of Proposition~\ref{prop:b=2step3} relies on the following estimates.
\begin{lemma}\label{lem:b=2.3} If $\gamma\le 0.64$, then for any $x\in [0,1/2]$, $(01)\not\sim_x (11)$.
\end{lemma}
\begin{proof}
The case $0\le x\le 2/5$ was treated in Lemma~\ref{lem:b=2.2}. Here we consider the case $x\in [2/5, 1/2]$. Let
$\textbf{u}=\{u_n\}_{n=1}^\infty$ and $\textbf{v}=\{v_n\}_{n=1}^\infty$ be such that $(u_1u_2)=(01)$ and $(v_1v_2)=(11)$.
Let $F(x)=-(2\pi)^{-1}(S(x, \textbf{u})-S(x, \textbf{v}))$ and let $P_n, Q_n$ be defined as above.
Then $$P_1=2\sin (\pi x)\ge 2\sin \frac{2\pi}{5}=\frac{\sqrt{10+2\sqrt{5}}}{2}>1.902,$$
and $$P_2=\cos \frac{\pi x}{2}-\sin \frac{\pi x}{2}>0.$$

{\bf Case 1.} $u_3=0$ or $v_3=1$.
Then
\begin{multline*}
P_3=(-1)^{u_3}\cos\frac{\pi x}{4}-(-1)^{v_3}\cos \frac{\pi (1+x)}{4}\ge \cos \frac{\pi (1+x)}{4}-\cos \frac{\pi x}{4}\\
\ge -2 \sin\frac{\pi}{8}\sin \frac{\pi(1+2x)}{8}>-\sqrt{2}\sin \frac{\pi}{8},
\end{multline*}
so
\begin{align*}
F(x)& \ge P_1+\gamma P_2+\gamma^2 P_3-\frac{2\gamma^3}{1-\gamma}\\
& > 1.902 -\sqrt{2}\cdot 0.64^2\sin\frac{\pi}{8}-\frac{2\cdot 0.64^3}{1-0.64}>0.
\end{align*}

{\bf Case 2.} $u_3=1$ and $v_3=0$. Then
\begin{multline*}
P_3=-\cos \frac{\pi (1+x)}{4}-\cos \frac{\pi x}{4}=-2\cos \frac{\pi}{8}\cos \frac{\pi (1+2x)}{8}\\
>-2\cos^2\frac{\pi}{8}=-1-\frac{\sqrt{2}}{2}>-1.708.
\end{multline*}

{\em Subcase 2.1.} $i_4=1$ and $j_4=0$. Then
$$P_4=-\sin \frac{\pi (2-x)}{8}-\sin \frac{\pi (x+3)}{8}=-2\sin \frac{5\pi}{16} \cos \frac{\pi (1+2x)}{16}>-2\sin \frac{5\pi}{16}>-1.663.$$
and $$P_5\ge -\sin \frac{\pi(2-x)}{16}-\sin \frac{\pi (x+3)}{16}=-2\sin \frac{5\pi}{32} \cos \frac{\pi (1+2x)}{32}>-2\sin \frac{5\pi}{32}>-0.942.$$
Thus
\begin{multline*}
F(x)\ge P_1+\gamma P_2+\gamma^2 P_3+\gamma^3 P_4+\gamma^4 P_5-\frac{2\gamma^5}{1-\gamma}\\
\ge 1.902 -0.64^2\cdot 1.708 - \cdot 0.64^3 \cdot 1.663 -0.942 \cdot 0.64^4 -\frac{2\cdot 0.64^5}{1-0.64}
>0.011.
\end{multline*}


{\em Subcase 2.2.} Either $i_4=0$ or $j_4=1$. Then
$$P_4\ge \sin \frac{\pi (2-x)}{8}-\sin \frac{\pi (3+x)}{8}=-2\sin \frac{\pi (1+2x)}{16}\cos \frac{5\pi}{16}> -2\sin \frac{\pi}{8}\cos \frac{5\pi}{16}>-0.556.$$
Thus
\begin{multline*}
F(x)\ge P_1+\gamma P_2 +\gamma^2 P_3 +\gamma^3 P_4 -\frac{2\gamma^4}{1-\gamma}\\
>1.902 -0.64^2\cdot 1.708-0.64^3\cdot 0.556-\frac{2\cdot 0.64^4}{1-0.64}>0.124.
\end{multline*}
\end{proof}

Summarizing the results obtained in Lemma~\ref{lem:b=2.2*} and Lemma~\ref{lem:b=2.3}, we have
\begin{lemma}\label{lem:b=2.3*}
Assume $\gamma\le 0.64$.
Then
\begin{enumerate}
\item[(1)]  For $x\in [1/5,2/5]$, $e(2,x)=1$;
\item[(2)] For $x\in [0,1/5)$, the only possible non-trivial pairs in $E(2,x)$ are $(10,11)$ and $(11,10)$;
\item[(3)] For $x\in (2/5, 1/2]$, the only possible non-trivial pairs in $E(2,x)$ are $(01,10)$ and $(10,01)$.
\end{enumerate}
\end{lemma}
\begin{proof} This follows immediately from Lemma~\ref{lem:b=2.2*} and Lemma~\ref{lem:b=2.3}.
\end{proof}
\begin{proof}[Proof of Proposition~\ref{prop:b=2step3}] Let $K=[1/5, 2/5]\cup [3/5, 4/5]$. Then the conditions in Lemma~\ref{lem:sqrt2} are satisfied with $q=2$ and suitable choice of $\eps,\delta$. Indeed,
\begin{itemize}
\item By Lemma~\ref{lem:b=2.3*}, for each $x\in [1/5,2/5]$, we have $e(2, x)=1$. So by Lemma~\ref{lem:symmetric} and Proposition~\ref{prop:compact} (3), condition (i) of Lemma~\ref{lem:sqrt2} is satsified with suitable choices of $(\eps, \delta)$;
\item By Lemma~\ref{lem:b=2.3*}, for each $x\in [0,1/5)$, $(10,11)$ and $(11,10)$ are the only non-trivial pairs in $E(2,x)$, and it is easily checked $x(10), x(11)\in K$. For $x\in (2/5, 1/2]$, by Lemma~\ref{lem:b=2.3*}, $(01,10)$ and $(10,01)$ are the only non-trivial pairs in $E(2,x)$ and it is easily checked that $x(10), x(01)\in K$. Thus by Lemma~\ref{lem:symmetric} and Proposition~\ref{prop:compact} (3), condition (ii) of Lemma~\ref{lem:sqrt2} is satsified.
\end{itemize}
Thus $\sigma (2)\le \sqrt{2}$.
\end{proof}

\subsection{When $\gamma^3 > \sqrt{2}/8$}
\begin{prop}\label{prop:b=2step4}
Assume $\gamma^2\le \sqrt{2}/4$. Then $\sigma(3)\le \sqrt{2}$.
\end{prop}
An immediate corollary of this proposition and Corollary~\ref{cor:b=2.3} is the following:
\begin{coro}\label{cor:b=2.4} If $\gamma^3> \sqrt{2}/8$, then $\sigma (q)< (2\gamma)^q$ for some $q\in\{1,2,3\}$.
\end{coro}
\begin{lemma}\label{lem:b=2.4}
Assume $\gamma^2\le \sqrt{2}/4$. If $(01u_3)\sim_x (10v_3)$ for some $x\in [0,1/2]$ then $x\in (17/36, 1/2]$, $u_3=1$ and $v_3=0$.
\end{lemma}
\begin{proof} Note that $\gamma <0.6$.
Let $\textbf{u}=01u_3\cdots$, $\textbf{v}=10v_3\cdots$ and let $F(x)=-(2\pi)^{-1}(S(x,\textbf{u})-S(x,\textbf{v}))$, $G(x)=F'(x)/\pi$. We continue to use the notation $x_n, y_n, P_n$ and $Q_n$.

If $x\le 0.45$, then as in the proof of Lemma~\ref{lem:b=2.1},
\begin{align*}
G(x)& \ge 2\cos (\pi x) -\frac{\gamma}{\sqrt{2}}\cos \left(\frac{\pi x}{2} +\frac{\pi}{4} \right)-\frac{\gamma^2}{2-\gamma}\\
& \ge 2 \cos (0.45\pi )- 0.6\cos (0.475\pi)/\sqrt{2}-0.36/1.4>0.02.
\end{align*}

Assume $(u_3,v_3)\not =(1,0)$. We shall show that for each $x\in [9/20, 1/2]$, $F(x)>0$. Indeed, in this case,
$P_1=2\sin (\pi x)$, $P_2=-\sin\frac{\pi x}{2}-\cos \frac{\pi x}{2}$ and
\begin{align*}
P_3 &\ge \cos \frac{\pi (1-x)}{4}-\cos \frac{\pi x}{4} && =-2\sin \frac{\pi}{8} \sin\frac{\pi (1-2x)}{8}\\
& \ge  -2\sin\frac{\pi}{8}\sin \frac{\pi}{40} && \ge -\frac{\pi^2}{160}.
\end{align*}
Thus
\begin{align*}
F(x)& \ge P_1+\gamma P_2+\gamma^2 P_3-\frac{2\gamma^3}{1-\gamma}\\
& \ge 2\sin (\pi x)-\gamma \left(\cos \frac{\pi x}{2}+\sin \frac{\pi x}{2}\right)-\frac{\gamma ^2\pi^2}{160}-\frac{2\gamma^3}{1-\gamma}\\
& \ge 2\sin (\pi x)- 0.6 \left(\cos \frac{\pi x}{2}+\sin \frac{\pi x}{2}\right)-0.023-1.08.
\end{align*}
For $x\in [9/20, 1/2]$, this give us
$$F(x)> 2\sin \frac{9\pi}{20} -0.6 \sqrt{2} -1.103>0.013.$$

Assume now $(u_3, v_3)=(1,0)$. We shall show that $G(x)>0$ for all $x\in [9/20, 17/36]$. Indeed, in this case,
$$Q_3=\cos \frac{\pi (2-x)}{4}-\cos \frac{\pi (1+x)}{4}=-2\sin \frac{3\pi}{8} \sin \frac{\pi (1-2x)}{8}> -\frac{\pi}{40}.$$
Thus
\begin{align*}
G(x)& =Q_1+\frac{\gamma}{2} Q_2 +\frac{\gamma^2}{4} Q_3-\frac{\gamma^3}{2-\gamma}\\
& \ge 2\cos (\pi x) -\frac{\gamma}{\sqrt{2}}\cos \left(\frac{\pi x}{2} +\frac{\pi}{4} \right)-\frac{\gamma^2}{4} \frac{\pi}{40}-\frac{\gamma^3}{4-2\gamma}\\
& \ge 2 \cos \frac{17\pi}{36} - \frac{0.6}{\sqrt{2}}\cos \frac{35 \pi}{72} -0.01-0.08\\
& > 0.174- 0.02- 0.01-0.08>0.
\end{align*}
\end{proof}

\begin{lemma}\label{lem:b=2.4*} Assume $\gamma<0.6$. If $x\in [0,1/2]$ and $\textbf{u},\textbf{v}$ are distinct elements in $\mathcal{A}^3$ such that $(\textbf{u},\textbf{v})\in E(3, x)$, then
either $$x\in [0,1/9) \text{ and }(\textbf{u},\textbf{v})\in\{(101,110), (110,101), (010,011), (011, 010)\};$$
or    $$x\in (17/36,1/2] \text{ and } (\textbf{u},\textbf{v})\in \{(011, 100), (100, 011)\}.$$
\end{lemma}
\begin{proof} Lemma~\ref{lem:b=2.4} particularly implies that $(01)\not\sim_x (10)$ for $x\in [2/5, 17/36]$. Together with Lemma~\ref{lem:b=2.3*}, it follows that
\begin{enumerate}
\item [(a)] $0\not\sim_x 1$ for $x\in [0, 17/36]$.
\end{enumerate}
By Lemma~\ref{lem:symmetric},
\begin{enumerate}
\item [(a')] $0\not\sim_x 1$ for $x\in [19/36, 1]$.
\end{enumerate}
Therefore, for $x\in [1/18, 1/5]$, $(10)\not\sim_x (11)$, since $x(1)\in [19/36, 3/5]$. Together with Lemma~\ref{lem:b=2.3*}, we obtain
\begin{itemize}
\item [(b)] For $x\in [1/18, 17/36]$, $e(2,x)=1$;
\item [(c)] For $x\in [0, 1/18)$, the only possible non-trivial pairs in $E(2, x)$ are $(10,11)$ or $(11,10)$;
\item [(d)] For $x\in (17/36, 1/2]$, the only possible non-trivial pairs in $E(2, x)$ are $(01,10)$ or $(10,01)$.
\end{itemize}

Consider $x\in [1/9, 17/36]$. If $(u_1u_2u_3)\sim_x (v_1v_2v_3)$ then by (b), $u_1u_2=v_1v_2$. Note that for any $u_1 u_2$, $x(u_1u_2)\not\in [17/36, 19/36]$. Thus by (a) (and (a')), we have $u_3=v_3$. This proves that $e(3, x)=1$.

Consider $x\in [0, 1/9)$ and $(u_1u_2u_3)\sim_x (v_1v_2v_3)$. Then by (a), $u_1=v_1$. If $u_1=0$ then $x(u_1)\in [0, 1/18)$, and $(u_2u_3)\sim_{x(u_1)} (v_2v_3)$, so by (c), $(u_2u_3, v_2v_3)\in \{(10,11), (11,10)\}$. If $u_1=1$, then $x(u_1)\in [1/2, 10/18]$, and $(u_2u_3)\sim_{x(u_1)} (v_2v_3)$, so by Lemma~\ref{lem:b=2.3*} (3), $(u_2u_3, v_2v_3)\in \{(01,10), (10, 01)\}$.

Consider $x\in [17/36, 1/2]$. Then by (d) and Lemma~\ref{lem:b=2.4}, the only possible non-trivial pairs in $E(3, x)$ are $(011,100)$ and $(100,011)$.
\end{proof}

\begin{proof}[Proof of Proposition~\ref{prop:b=2step4}]
Let $K=[1/9, 17/36]\cup [36/19, 8/9]$. The conditions in Lemma~\ref{lem:sqrt2} are satisfied with $q=3$ and suitable choices of $(\eps,\delta)$. Indeed,
putting $L_1=(17/36, 19/36)$, $L_2=[0,1/9)\cup (8/9, 1)$.
by Lemma~\ref{lem:b=2.4*}, Lemma~\ref{lem:symmetric} and Proposition~\ref{prop:compact} (3), there exist $\eps>0$ and $\delta>0$ such that
\begin{itemize}
\item $e(3, x;\eps,\delta)=1$ for all $x\in K$.
\item For $x\in L_1$, $e(3, x;\eps,\delta)\le 2$ and the only non-trivial elements of $\mathcal{A}^3$ which appears in a non-trivial pair of $E(3, x)$ are
 $(011)$ and $(100)$ for which we have $x(011), x(100)\in K$.
\item For $x\in [0, 1/9)$, $e(3,x;\eps,\delta)\le 2$ and the only elements of $\mathcal{A}^3$ which appear in non-trivial pairs of $E(3, x)$ are $(010), (011), (101)$ and $(110)$ for which $x(010)$, $x(011)$, $x(101)$, $x(110)\in K$. By Lemma~\ref{lem:symmetric}, similar properties hold for $x\in (8/9,1)$.
\end{itemize}
By Lemma~\ref{lem:sqrt2}, we have $\sigma (3)\le \sqrt{2}$.
\end{proof}

\subsection{Last Step: When $\gamma^3\le \sqrt{2}/8$}
\begin{prop}\label{prop:b=2step5} Assume $\gamma^3\le \sqrt{2}/8$. Then $e(1)=1$.
\end{prop}
\begin{proof} It suffices to show that
\begin{equation}\label{eqn:b=2.5}
(011)\not\sim_x (100)\text{ for }x\in [17/36,1/2].
\end{equation}
Indeed, 
by Lemma~\ref{lem:b=2.4*}, it follows that $0\not\sim_x 1$ for all $x\in [0, 1/2]$. By Lemma~\ref{lem:symmetric}, $0\not\sim_x 1$ also holds for $x\in [1/2, 1]$.

To prove (\ref{eqn:b=2.5}), let $\textbf{u}=(011\cdots)$, $\textbf{v}=(100\cdots)$ and $F(x)=-(2\pi)^{-1}(S(x,\textbf{u})-S(x,\textbf{v}))$.
We shall prove that $F(x)>0$ holds for each $x\in [17/36,1/2]$. We shall continue to use the notation $x_n, y_n$, $P_n$ and $Q_n$ as above.
Then $P_1=2\sin (\pi x)$, $P_2=-\sin \frac{\pi x}{2}-\cos \frac{\pi x}{2}$ and
$$P_3=-\sin \frac{\pi(2-x)}{4}-\sin \frac{\pi(1+x)}{4}=-2\sin \frac{3\pi}{8}\cos \frac{\pi (1-2x)}{8}.$$

Put $$g(x)=P_1+\gamma P_2+\gamma^2 P_3.$$
Using $\gamma\le 0.562$, it is easy to check that $g''<0$ on $[9/20, 1/2]$. By calculation, $$g(9/20)>0.602\text{ and }g(1/2)>0.624.$$ It follows that
$$g(x)> 0.602 \text{ for any } x\in [17/36, 1/2].$$

{\bf Case 1.} $i_4=0$ or $j_4=1$.
Then $$P_4\ge -\left(\sin \frac{\pi (2-x)}{8}-\sin \frac{\pi (1+x)}{8}\right)\ge -2\sin\frac{\pi (1-2x)}{8}\ge -2\sin \frac{\pi}{144}>-0.05 .$$
Then
\begin{align*}
F(x)& \ge g(x) +\gamma^3 P_4 -\frac{2\gamma^4}{1-\gamma} \\
& >0.602-0.05\gamma^3 - 2\gamma^4/(1-\gamma)>0.
\end{align*}

{\bf Case 2.} $i_4=1$ and $j_4=0$. Then
$$P_4= -\left(\sin\frac{\pi (2-x)}{8}+\sin\frac{\pi (1+x)}{8}\right)\ge -2 \sin \frac{3\pi}{16}> -1.112$$
and
$$P_5\ge -\left(\sin\frac{\pi (2-x)}{16}+\sin\frac{\pi (1+x)}{16}\right)\ge -2\sin \frac{3\pi}{32} >-0.581.$$
Thus
$$F(x)\ge g(x)+\gamma^3 P_4+\gamma^4 P_5-\frac{2\gamma^5}{1-\gamma}> 0.602-1.112\gamma^3-0.581\gamma^4-\frac{2\gamma^5}{1-\gamma}>0.09.$$
\end{proof}

\section{Appendix: A proof of Ledrappier's theorem}
This appendix is devoted to a proof of Ledrappier's theorem under a further assumption that $\phi'$ is continuous (for simplicity). The proof is motivated by the original proof given in~\cite{L} and also the recent paper~\cite{K}.

Let $b\ge 2$ be an integer and $\lambda\in (1/b,1)$ and let $f(x)=\sum_{n=0}^\infty \lambda^n \phi(b^n x)$.
We use $z$ to denote a point in $\R^2$ and $B(z,r)$ denote the open ball in $\R^2$ centered at $z$ and of radius $r$. We assume that $\dim(m_x)=1$ holds for Lebesgue a.e. $x\in [0,1)$, which means
for Lebesgue a.e. $x\in [0,1)$ and $\mathbb{P}$-a.e. $\textbf{u}\in \mathcal{A}^{\Z^+}$,
$$\lim_{r\to 0} \frac{\log\mathbb{P}\left(\{\textbf{v}: |S(x,\textbf{u})-S(x,\textbf{v})|\le r\}\right)}{\log r}=1.$$
Let $\mu$ be the pushforward of the Lebesgue measure on $[0,1)$ under the map $x\mapsto (x, f(x))$. Let
$$\underline{d}(\mu, z)=\liminf_{r\to 0}\frac{\log \mu(B(z,r))}{\log r}$$
and $$\underline{D}=\text{essinf}\,\, \underline{d} (\mu, z).$$
We shall prove that
$$\underline{D}\ge D=2+\frac{\log\lambda}{\log b}.$$
This is enough to conclude that the Hausdorff dimension of the graph of $f$ is $D$. Indeed, by the mass distribution principle, it implies that the Hausdorff dimension is at least $D$. On the other hand, it is easy to check that $f$ is  a $C^{2-D}$ function which implies that the Hausdorff dimension is at most $D$ (see for example Theorem 8.1 of \cite{F}).

\subsection{Telescope}
Define $\Phi: [0,1)\times \R \to [0,1)\times \R$ as
$$\Phi(x,y)=(bx \mod 1, (y-\phi(x))/\lambda).$$
Define $G: [0,1)\times \R\times \mathcal{A}^{\Z^+}\to [0,1)\times \mathcal{A}^{\Z^+}$ as
$$G(x,y, \textbf{u})=(\Phi(x,y), u_0\textbf{u}), \text{ if } bx\in [u_0, u_0+1).$$

The graph of $f$ is an invariant repeller of the expanding map $\Phi$. We shall use neighborhoods bounded by unstable manifolds.
For each $z_0=(x_0,y_0)\in\R^2$ and $\textbf{u}\in\mathcal{A}^{\Z^+}$, let $\ell_{z_0,\textbf{u}}(x)$
denote the unique solution of the initial value problem:
$$y'=-\gamma S(x,\textbf{u}), y(x_0)=y_0.$$
These curves are strong unstable manifolds of $\Phi$ and they satisfy the following property: for $z=(x,y),z_0=(x_0,y_0)\in [u_0/b, (u_0+1)/b)\times \R$, $u_0\in\mathcal{A}$,
$$\Phi(x, \ell_{z_0, \textbf{u}}(x))=(bx-u_0, \ell_{\Phi(z_0), u_0\textbf{u}}(bx-u_0)).$$

For $z_0=(x_0,y_0)\in [0,1)\times \R$, $\textbf{u}\in\mathcal{A}^{\Z^+}$ and $\delta_1,\delta_2>0$, let
$$Q(z_0, \textbf{u},\delta_1,\delta_2)=\{(x,y): x\in [0,1), |x-x_0|\le \delta_1, |y-\ell_{z_0,\textbf{u}}(x)|\le \delta_2\}.$$

The following observation was taken from ~\cite{K}.
\begin{lemma}[Telescope]\label{lem:tele} Let $\{(z_i,\textbf{u}_i)\}_{i=0}^n$ be a $G$-orbit and let $x_i$ denote the first coordinate of $z_i$. For any $\delta_1,\delta_2>0$, if
$\delta_1\le x_n<1-\delta_1$, then
$$\mu(Q(z_0,\textbf{u}, \delta_1 b^{-n}, \delta_2 \lambda^n))=b^{-n} \mu(Q(z_n,\textbf{u}_n, \delta_1,\delta_2)).$$
\end{lemma}
\begin{proof} Let $J_i=[x_i-\delta_1 b^{i-n}, x_i+\delta_1 b^{i-n}]$, $Q_i=Q(z_i,\textbf{u}_i, \delta_1 b^{i-n},\delta_2 \lambda^{n-i})$ and let $E_i=\{x\in J_i: (x, f(x))\in Q_i\}$. Then $\mu(Q_i)=|E_i|$. Under the assumption $\delta_1\le x_n<1-\delta_1$, $Q_0$ is mapped onto $Q_n$ diffemorphically under $\Phi^n$. Thus $J_0$ is mapped onto $J_n$ and $E_0$ is mapped onto $E_n$ diffeomorphically under the linear map $x\mapsto b^n x$. Thus $|E_0|=b^{-n}|E_n|$.
\end{proof}

\subsection{A version of Marstrand's estimate}
Fix a constant $t\in (1/(1+\alpha),1)$.
\begin{prop} \label{prop:mars}
For $\mu\times \mathbb{P}$-a.e. $(z_0,\textbf{u})$,
\begin{equation}
\liminf_{r\to 0}\frac{\log \mu(Q(z_0,\textbf{u}, r^t, r))}{\log r} \ge 1+t(\underline{D}-1).
\end{equation}
\end{prop}
\begin{proof}
It suffices to prove that for each $\xi>0$ and $\eta>0$, there is a subset $\Sigma$ of $[0,1)\times \R\times \mathcal{A}^{\Z^+}$ with
$(\mu\times \mathbb{P})(\Sigma)>1-\eta$ such that
\begin{equation}\label{eqn:marspart}
\liminf_{r\to 0} \frac{\log \mu(Q(z_0,\textbf{u}, r^t, r))}{\log r} \ge 1+t(\underline{D}-1)-3\xi
\end{equation}
holds for all $(z_0,\textbf{u})\in \Sigma$.
By Egoroff's theorem, we can choose $\Sigma$ with $(\mu\times\mathbb{P})(\Sigma)>1-\eta$ for which there is $r_0>0$ such that
for each $(z_0,\textbf{u})\in \Sigma$,
\begin{enumerate}
\item[(S1)] $\mathbb{P}\left(\left\{\textbf{v}: |S(x_0,\textbf{u})-S(x_0,\textbf{v})|\le r\right\}\right)\le r^{1-\xi}$ for each $0< r\le r_0$, where $x_0$ is the first coordinate of $z_0$;
\item[(S2)] $\mu(B(z_0, r))\le r^{\underline{D}-\xi}$ for each $0<r\le r_0$.
\end{enumerate}
In the following we shall prove that for $r>0$ small enough,
\begin{equation}\label{eqn:marsint}
\int_{\textbf{u}: (z_0,\textbf{u})\in \Sigma} \mu(Q(z_0, \textbf{u}, r^t, r))d\mathbb{P}\le r^{1+t(\underline{D}-1)-2\xi},
\end{equation}
holds for every $z_0\in [0,1)\times \R$.
This is enough to conclude the proof. Indeed, let $\tau\in (0,1)$ be an arbitrary constant. Then by (\ref{eqn:marsint}), there is $N$ such that for $n>N$,
$$\mathbb{P}\left(\left\{\textbf{u}: (z_0,\textbf{u})\in \Sigma, \mu(Q(z_0,\textbf{u},\tau^{nt}, \tau^n))> (\tau^n)^{1+t(\underline{D}-1)-3\xi}\right\}\right)\le \tau^{n\xi}$$
holds for every $z_0\in [0,1)\times \R$.
By Fubini's theorem, this implies that
$$\mu\times\mathbb{P}\left(\left\{(z_0,\textbf{u})\in \Sigma: \mu(Q(z_0,\textbf{u},\tau^{nt}, \tau^n))> (\tau^n)^{1+t(\underline{D}-1)-3\xi}\right\}\right)\le\tau^{n\xi}.$$
By Borel-Cantelli, it follows that for almost every $(z_0,\textbf{u}) \in \Sigma$,
$\mu(Q(z_0,\textbf{u},\tau^{nt},\tau^n))\le (\tau^n)^{1+t(\underline{D}-1)-3\xi}$ holds for all $n$ large enough. The inequality (\ref{eqn:marspart}) follows.

Let us now prove (\ref{eqn:marsint}). We first prove

{\bf Claim.} Provided that $r>0$ is small enough, for every $z_0, z\in [0,1)\times \R$, we have
\begin{equation}\label{eqn:Pu}
\mathbb{P}(\left\{\textbf{u}: (z_0,\textbf{u})\in \Sigma, z\in Q(z_0,\textbf{u}, r^t, r)\right\})\le C_1 \left(\frac{r}{|z-z_0|}\right)^{1-2\xi},
\end{equation}
where $C_1>0$ is a constant.

To prove this claim, let $z=(x,y)$, $z_0=(x_0,y_0)$ and $h(x)=\ell_{z_0,\textbf{u}}(x)$. Then $h(x)$ is $C^{1+\alpha}$ with uniformly bounded norm. So
\begin{multline*}
|y-y_0+\gamma S(x_0,\textbf{u}) (x-x_0)|\le |y-h(x)|+|h(x)-h(x_0)-h'(x_0)(x-x_0)|\\
\le r+|\int_{x_0}^x (h'(s)-h'(x_0))ds |\le r+ C r^{t(1+\alpha)}< 2r,
\end{multline*}
provided that $r$ is small enough.
Thus
$$\{S(x_0,\textbf{u}):(z_0,\textbf{u})\in \Sigma, z\in Q(z_0, \textbf{u}, r^t, r)\}$$
is contained in an interval of length $2r/(\gamma |x-x_0|)$. Since
$$|z-z_0|\le |x-x_0|+|y-y_0|\le (1+\gamma |S(x_0,\textbf{u})|)|x-x_0|+2r,$$
and $|S(x,\textbf{u})|$ is uniformly bounded, the inequality (\ref{eqn:Pu}) follows from the property (S1).
Note that if $2r/(\gamma |x-x_0|)>r_0$, then $r/|z-z_0|$ is bounded away from zero, so (\ref{eqn:Pu}) holds for sufficiently large $C_1$, since the left hand side of this inequality does not exceed one.

We continue the proof of (\ref{eqn:marsint}). Note that there is a constant $C_2>0$ such that for every $r>0$ and any $z_0\in [0,1)\times \R$,
$$\bigcup_{\textbf{u}\in \mathcal{A}^{\Z^+}} Q(z_0, \textbf{u}, r^t, r)\subset B(z_0, C_2r^t).$$
Of course we may assume there is $\textbf{u}$ such that $(z_0,\textbf{u})\in \Sigma$. Thus for $R>0$ small enough, we may apply (S2) and obtain
\begin{align*}
\int_{B(z_0,R)} \frac{d\mu(z)}{|z-z_0|^{1-2\xi}}& =\sum_{n=0}^\infty \int_{e^{-n-1}R\le |z-z_0|< e^{-n} R} \frac{d\mu(z)}{|z-z_0|^{1-2\xi}}\\
& \le \sum_{n=0}^\infty \frac{\mu(B(z_0, e^{-n}R))}{(e^{-n-1}R)^{1-2\xi}}\le \sum_{n=0}^\infty \frac{(e^{-n}R)^{\underline{D}-\xi}}{(e^{-n-1}R)^{1-2\xi}}\\
& =C(\xi) R^{\underline{D}-1+\xi},
\end{align*}
where $C(\xi)$ is a constant depending on $\xi$ and $\underline{D}$.
By Fubini's theorem,
\begin{align*}
& \int_{\textbf{u}: (z_0,\textbf{u})\in \Sigma} \mu (Q(z_0, \textbf{u}, r^t, r) d\mathbb{P}(\textbf{u})\\
=& \int_{\textbf{u}: (z_0,\textbf{u})\in \Sigma}\int_{\R^2} 1_{Q(z_0, \textbf{u}, r^t, r)}(z) d\mu (z) d\mathbb{P}(\textbf{u})\\
= & \int_{B(z_0, C_2r^t)} \mathbb{P}\left(\left\{\textbf{u}: (z_0,\textbf{u})\in \Sigma, z\in Q(z_0,\textbf{u}, r^t, r)\right\}\right) d\mu(z)\\
& \le C_1 r^{1-2\xi} \int_{B(z_0, C_2 r^t)} \frac{d\mu(z)}{|z-z_0|^{1-2\xi}} \\
& \le C' r^{1+t(\underline{D}-1)-(2-t)\xi}< r^{1+t(\underline{D}-1)-2\xi},
\end{align*}
provided that $r$ is small enough.
\end{proof}

We are ready to complete the proof of Ledrappier's theorem.
For any $\xi>0$, $\eta>0$, by Proposition~\ref{prop:mars} and Egroff's theorem, we can pick up a subset $\Sigma$ of $\R^2\times \mathcal{A}^{\Z^+}$ and a constant $r_*>0$ such that
$(\mu\times \mathbb{P})(\Sigma)>1-3\eta$ and such that for each $(z,\textbf{u})\in \Sigma$,
$$\mu(Q(z,\textbf{u}, r^t, r))\le r^{1+t(\underline{D}-1)-\xi} \text{ for each } 0<r<r_*.$$
We may further assume that $\Sigma\subset [\eta, 1-\eta]\times \R\times \mathcal{A}^{\Z^+}$.

Note that $\mu\times \mathbb{P}$ is an ergodic invariant measure for the map
$G$. By Birkhorff's Ergodic Theorem, for almost every $(z_0,\textbf{u}_0)$, there is an increasing sequence $\{n_k\}_{k=1}^\infty$ of positive integers such that
$G^{n_k}(z_0,\textbf{u}_0)\in \Sigma$ and
\begin{equation}\label{eqn:nkdensity}
\liminf_{k\to\infty} n_k/n_{k+1} > 1-3\eta.
\end{equation}

For each $n=1,2,\ldots$, put $\delta_n=\gamma^{nt/(1-t)}b^{-n}$, $r_n=\gamma^{n/(1-t)}$, so that
$$r_n=\delta_n \lambda^{-n}, \text{ and } r_n^t=\delta_n b^n.$$
Let us prove that for $k$ sufficiently large,
\begin{equation}\label{eqn:mudeltan}
\frac{\log \mu(Q(z_0,\textbf{u}_0, \delta_{n_k}, \delta_{n_k}))}{\log \delta_{n_k}}\ge \underline{D}+(D-\underline{D})A_1-A_2\xi,
\end{equation}
where $A_1, A_2$ are positive constants depending only on $\lambda$ and $b$. 

Indeed, by Lemma~\ref{lem:tele}, for $k$ large enough,
$$\mu(Q(z_0,\textbf{u}_0, \delta_{n_k}, \delta_{n_k}))=\frac{\mu(Q(G^{n_k}(z_0,\textbf{u}_0),r_{n_k}^t, r_{n_k}))}{b^{n_k}}
\le \frac{r_{n_k}^{1+t(\underline{D}-1)-\xi}}{b^{n_k}}.$$
Using definition of $r_n$ and $\delta_n$, this gives us
$$\mu(Q(z_0,\textbf{u}_0,\delta_{n_k},\delta_{n_k}))\le \delta_{n_k}^{\underline{D}}\times (b^{-n_k})^{D-\underline{D}} r_{n_k}^{-\xi},$$
Thus (\ref{eqn:mudeltan}) holds with $A_1=\log b/(\log b+t\log \gamma^{-1}/(1-t))$ and $A_2=\log\gamma/(t\log\gamma +(1-t)\log b^{-1})$.

By (\ref{eqn:nkdensity}), for each $n$ large enough, there is $k$ such that $(1-3\eta)n_k<n_{k-1}<n\le n_k$. It follows that
$$\liminf_{n\to\infty}\frac{\log \mu(Q(z_0,\textbf{u}_0,\delta_n, \delta_n))}{\log\delta_n}\ge (1-3\eta) (\underline{D}+(D-\underline{D})A_1-A_2\xi).$$
Since $\ell_{x_0,\textbf{u}_0}$ is a smooth curve, there exists $\kappa\in (0,1)$ such that $Q(z_0,\textbf{u}_0,\delta_k,\delta_k)$ contains
$B(z_0, \kappa\delta_k)$ for each $k$. Therefore,
$$\underline{d}(\mu, z_0)=\liminf_{n\to\infty}\frac{\log \mu(Q(z_0,\textbf{u}_0,\delta_n, \delta_n))}{\log\delta_n}\ge (1-3\eta) (\underline{D}+(D-\underline{D})A_1-A_2\xi).$$
Since this estimate holds for $\mu$-a.e. $z_0$, we obtain
$$\underline{D}\ge (1-3\eta)\left(\underline{D}+ A_1(D-\underline{D})-A_2\xi\right).$$
As $\xi, \eta$ can be chosen arbitrarily small, we conclude
$$\underline{D}\ge \underline{D}+ A_1(D-\underline{D}),$$
which means $\underline{D}\ge D$, as desired.

\bibliographystyle{plain}             

\end{document}